\documentclass{amsart}
\usepackage{amssymb}
\usepackage{latexsym}
\usepackage{euscript}

\usepackage{graphicx}
\usepackage{lmodern}
\usepackage{pifont}

\usepackage[normalem]{ulem}

\usepackage{amsmath,amscd,amssymb,amsthm,amsfonts,color}
\usepackage{latexsym}
\usepackage{euscript}

\usepackage[english]{babel}
\usepackage[T1]{fontenc}
\usepackage[utf8]{inputenc}
\usepackage{xspace}
\usepackage{xcolor}

\usepackage{lmodern}
\usepackage{amsbsy}
\usepackage{url}
\usepackage{graphicx}
\usepackage{enumerate}

\newtheorem{theorem}{Theorem}[section]
\newtheorem{proposition}[theorem]{Proposition}
\newtheorem{corollary}[theorem]{Corollary}
\newtheorem{lemma}[theorem]{Lemma}

\newtheorem{remark}[theorem]{Remark}
\newtheorem{definition}[theorem]{Definition}

\usepackage[colorlinks=true,linkcolor=blue,citecolor=green]{hyperref}
\usepackage[capitalise]{cleveref}

\newcommand{\la}{\langle}
\newcommand{\ra}{\rangle}

\newcommand{\n}{\|}
\newcommand{\ifaf}{\Leftrightarrow}
\newcommand{\vide}{\emptyset}

\newcommand{\h}{\mathfrak{H}}
\newcommand{\kk}{\mathfrak{K}}
\newcommand{\lohk}{LO(\h,\kk)}
\newcommand{\loh}{LO(\h)}
\newcommand{\cohk}{CO(\h,\kk)}
\newcommand{\coh}{CO(\h)}

\newcommand{\mul}{\mathrm{mul} \, }

\newcommand{\kerr}{\mathrm{ker}\,}
\newcommand{\kerT}{\mathrm{ker}  \, T}

\newcommand{\ran}{\mathrm{ran}\,}
\newcommand{\ranT}{\mathrm{ran}  \, T }

\newcommand{\ranbar}{\overline{\mathrm{ran}}\,}
\newcommand{\ranbarT}{\mathrm{\overline{ran}}\, T}
\newcommand{\ranbarA}{\mathrm{\overline{ran}}\, A}

\newcommand{\bhk}{B(\h,\kk)}

\newcommand{\bh}{B(\h)}
\newcommand{\bk}{B(\kk)}

\newcommand{\dombar}{\overline{\mathrm{dom}}\,}
\newcommand{\D}{\mathrm{dom}\,}
\newcommand{\DT}{\mathrm{dom}\,T }
\newcommand{\DB}{\mathrm{dom}\,B }

\newcommand{\DA}{\mathrm{dom}  \,A }

\newcommand{\adj}{^*}
\newcommand{\inv}{^{-1}}

\newcommand{\ts}{T_s}
\newcommand{\clo}{^{**}}
\newcommand{\tstar}{T^{*}}

\newcommand{\half}{^{\frac{1}{2}}}
\newcommand{\mhalf}{^{-\frac{1}{2}}}

\newcommand{\halfltsb}{(\lambda  T^*B)^{\frac{1}{2}}}
\newcommand{\halftsb}{( T^*B)^{\frac{1}{2}}}
\newcommand{\halfbt}{( B^* T)^{\frac{1}{2}}}

\newcommand{\halfltsbs}{ (\lambda \,  \tstar B)^{\half}_s }
\newcommand{\halftsbs}{ ( \tstar B)^{\half}_s }

\newcommand{\orth}{^\perp}

\newcommand{\N}{\mathbb{N}}
\newcommand{\C}{\mathbb{C}}

\newcommand{\R}{\mathbb{R}}
\newcommand{\rplus}{\mathbb{R^+}}

\newcommand{\ppa}{P_{\ker( T^*B)^{\perp}}}

\newcommand{\inclu}{\subseteq}
\newcommand{\converge}{\underset{n \to + \infty} \longrightarrow}

\newcommand{\lplusk}{B^{+}(\kk)}
\newcommand{\lplus}{B^{+}(\h)}
\newcommand{\gplus}{GL^{+}(\h)}

\newcommand{\mathscr}{\mathbf}

\newcommand{\ldeux}{\mathcal{L}_l^{+2}(\h)}

    \newcommand{\lldeux}{\mathcal{L}_{{l}^{\tiny-}}^{+2}(\h)}
    
\newcommand{\bardt}{\overline{\D}T}

\newcommand{\rest}{\upharpoonright}

\newcommand{\bc}{\begin{center}}
\newcommand{\ec}{\end{center}}

\numberwithin{equation}{section}
\makeatletter
\renewenvironment{proof}[1][\proofname]{%
   \par\pushQED{\qed}\normalfont%
   \topsep6\p@\@plus6\p@\relax
   \trivlist\item[\hskip\labelsep\bfseries#1\@addpunct{.}]%
   \ignorespaces
}{%
   \popQED\endtrivlist\@endpefalse
}
\makeatother

\renewcommand{\labelenumi}{(\arabic{enumi})}

\begin{document}

\title{Product of  nonnegative selfadjoint operators \\ in unbounded settings}

%

\author{Yosra~Barkaoui}
\address{Department of Mathematics and Statistics \\
University of Vaasa \\
P.O. Box 700, 65101 Vaasa \\
Finland}
\email{yosra.barkaoui@uwasa.fi \\ yosrabarkaoui@gmail.com}

\author{Seppo~Hassi}

\address{Department of Mathematics and Statistics \\
University of Vaasa \\
P.O. Box 700, 65101 Vaasa \\
Finland}
\email{sha@uwasa.fi}

\keywords{Nonnegative operator,  operator inequalities,
factorization of operators, quasi-affinity, quasi-similarity, linear relations}

\subjclass{47A62,  47B02, 47B25,47A06}


\maketitle

\textbf{\abstractname}
In this paper, necessary and sufficient conditions are established  for the factorization of a closed, in general, unbounded  operator  $ T=AB$ into a product of two nonnegative selfadjoint operators $A$ and $B.$ Already the special case, where $A$ or $B$ is bounded, leads to new results and is of wider interest,  since the problem is connected to the notion of similarity of the operator $T$ to a selfadjoint one, but, in fact, goes beyond this case.    It is proved that this subclass of operators can be characterized not only by means of quasi-affinity of $\tstar$  to an operator $S=S\adj \geq 0$, but also via Sebestyén inequality, a result known in the   setting of bounded operators $T.$ Another subclass of operators $T,$ where $A$ or $B$ has a bounded inverse, leads to a similar analysis. This gives rise to a   reversed version of   Sebestyén  inequality which is introduced in the present paper.    It is   shown that this second subclass, where $A\inv$ or $B\inv$ is bounded, can be characterized in a similar way by means of quasi-affinity of $T,$ rather that $\tstar,$  to an operator $S=S\adj \geq 0$.  Furthermore,  the connection between these two classes and   weak-similarity as well as quasi-similarity  to  some $S =S\adj \geq 0$ is investigated. Finally, the special case where $ S$ is bounded is considered.\\ \\


\section{Introduction}
\hspace*{0.3cm} In $2021$   M. Contino, M. A. Dritschel, A. Maestripieri, and S. Marcantognini \cite{product2021} (see also \cite{gustavo2013products}) showed  that similarity to a bounded positive operator is no longer sufficient to characterize the product of two positive bounded operators in the settings of infinite-dimensional complex Hilbert space, contrary to that of finite-dimension; see   \cite{positiveprodmatrices}. More precisely, for a bounded operator $T \in B(\h)$ they established the following characterization for similarity:
 \begin{align} \label{introo1}
\begin{split}
   T \text{ is similar to a}  & \text{ positive operator } \\
    & \Updownarrow \\
       T= AB \text{ with }  A, B \in \lplus & \text{ and, in addition, } A   \text{ or $B$ is invertible},
\end{split}
\end{align}
where $\lplus$ stands for the set of all bounded nonnegative operators on $\h;$ see \cite[Theorem 3.1]{product2021}. This result remains   true for unbounded operators $T;$ cf. Proposition \ref{theosimilarscalar}.
Even weaker conditions than similarity, such as quasi-similarity and quasi-affinity have also proven to be  insufficient  to fully characterize such a product. Instead, the product representation $T=AB,$ $ A,B \in \lplus$ was characterized by means of  Sebestyén inequality
\cite{sebestyen1983restrictions}
 as follows:
 \begin{equation} \label{introo2}
T= A B \qquad \ifaf \qquad   T \tstar \leq X \tstar \text{ for some } X \in \lplus;
\end{equation}
see \cite[Theorem 4.5]{product2021}.
Hence, a natural approach to improve the above results is either to pursue weaker concepts than quasi-affinity or to relax certain conditions on $T.$ \\
 \hspace*{0.4cm} One of the main purposes in the present paper is to investigate these questions and to extend the above results to the setting of unbounded operators $T.$ More precisely, a complete study is first carried out when a closed operator $T$ belongs to the following class of operators:
\begin{eqnarray*}
     \ldeux=  \Big\{T=AB; \ A \in \lplus \text{ and } B = B\adj \geq 0   \Big\}, 
\end{eqnarray*}
where $B$ is in general unbounded. It will be seen in \cref{section2}
            that every element of   $\ldeux$ satisfies an equality analogous to the one appearing in    $\eqref{introo2}.$ More generally,  for closed operators  $T$ and $B$ such that $\tstar B$ is selfadjoint,    Sebestyén theorem \cite{sebestyen1983restrictions} is generalized to the unbounded context as follows:
                 \begin{equation} \label{introo2bt}
                 X \overline{B_0} \inclu T \text{ for some } X \in \lplus \  \ifaf \  T^*T \leq \lambda\, T^*B,
                \end{equation}
               for the restriction  $B_0:=B \rest \D \tstar B$ of $B;$ cf. Theorem \ref{firsttheo01}. In the unbounded setting the restriction $B_0$ appears naturally, and, in fact, due to the equality
               $$
               \tstar B_0 = \tstar \overline{B_0}= \tstar B
               $$
               the equivalence in \eqref{introo2bt} can restated just with $B_0.$ Obviously, in the particular case where $\D \tstar B$ is a core for $B,$ i.e., $\overline{B_0}= B$, \eqref{introo2bt} is    instead stated for $B.$ This covers the bounded setting in which \eqref{introo2} is true for $B \in \lplus$ and the equivalence  \eqref{introo2bt} holds with equality  $T =XB.$  However, for the unbounded setting where $B \neq \overline{B_0},$ it is necessary to consider further conditions  including
               $B\adj T = \tstar B$  in order to state  \eqref{introo2bt} for $B;$ see Proposition \ref{sebestyenB}.

               The inclusion in \eqref{introo2bt} represents a good motivation for describing the connection between  the class $\ldeux$ and  the notion of quasi-affinity to a nonnegative selfadjoint operator. Recall from \cite[Definition 2.2]{unboundedaffinity} that  $T$ is said to be   \emph{quasi-affine}  to some operator  $S$ if there exists an injective $G \in B(\h)$ such that $\ranbar G =\h$ and the following inclusion holds:
               \begin{equation}\label{quasiintrodef}
                GT \inclu SG.
               \end{equation}
              In the bounded case,   treated in  \cite[Proposition 3.8]{product2021}, one can observe that the inclusion in  \eqref{quasiintrodef} is equivalent to
                \begin{equation}\label{snee}
                 S= \overline{GT G\inv}= (G^{-1})\adj \tstar G\adj.
                \end{equation}
                However, \eqref{snee} need not hold anymore  in the  unbounded setting and this motivates the investigation of a possible connection between quasi-affinity to $S=S\adj \geq 0$  and the existence  of nonnegative selfadjoint extensions of $G T G\inv,$ which in turn leads to the following characterization given in   Proposition  \ref{inclusionnfs}
\begin{align}
T \supseteq  AB \in \ldeux & \text{  with } \ranbarA =\h
 \nonumber \\
& \Updownarrow \label{introm01}  \\
\tstar \text{ is quasi-affine to } & S = S\adj \geq 0 \nonumber.
\end{align}
Motivated by \eqref{introo2bt}, this induces the following new characterization of Sebestyén inequality by means of quasi-affinity to some $S=S\adj \geq 0:$
\begin{eqnarray}
&\tstar T \leq  \lambda \tstar B  \text{ with }   \DT \inclu \D B   \text{ for some }  \lambda \geq 0, \ B= B\adj =\overline{B_0}\geq 0
  \nonumber  \\
&\Updownarrow \label{introm02} \\
& T =  A\, \overline{B_0} \in \ldeux \quad \text{with } \ranbarA = \h  \nonumber \\
&\Updownarrow \nonumber \\
&\tstar \text{ is } G\text{-quasi-affine to } S = S^* \geq 0  \text{ with }    \DT \inclu \D \overline{B_F \rest \D(\tstar B_F}) \nonumber\\
 & \hspace*{-6.8cm} \text{and }  B_F=G\inv S\half \,  \overline{S\half  {(G\inv)}\adj};\nonumber
\end{eqnarray}
see Theorem \ref{coralama}.\\
\hspace*{0.4cm} The present setting of unbounded operators leads to further generalisations of the equivalences  in  \eqref{introo2bt} and \eqref{introm01}. In particular, the next goal in this paper is to investigate the reversed    inequality
                 \begin{equation} \label{reverintro}
                \tstar T \geq  \eta AT,  \qquad    \eta >0,
                \end{equation}
       and prove analogs for the characterizations in \eqref{introo2bt} and \eqref{introm01}; see   Theorem \ref{corsebestyeninclusion} and Corollary \ref{inversesebestyen}. The idea to get further characterizations here is to make connection to the initial Sebestyén inequality \eqref{introo2bt} by taking inverses in the operator inequality \eqref{reverintro}. This has motivated   a further  generalisation of the above results  to the case of nondensely defined operators as well as multivalued linear operators (linear relations) in    Theorem \ref{relationtheo}. \\
               \hspace*{0.4cm} For the reversed  inequality \eqref{reverintro},  quasi-affinity of $T$, rather than $\tstar$, to $S$ arises
               and leads to a new  class different from $\ldeux$ defined by
                       \begin{align*}
                            \lldeux=
                             &  \{ T= BA , \ B\inv \in \lplus \text{ and }   A = A\adj \geq 0 \}.
                 \end{align*}
                 In fact,
                 Theorem \ref{TBAsurjective} shows that:
                \begin{equation}\label{introm3}
                   T \inclu B A  \in \lldeux \ifaf T \text{ is  quasi-affine to some } S=S\adj \geq 0.
                \end{equation}
 In particular, if $T$ is $G$-quasi-affine to $S$  such that $\rho(\overline{ G\adj S\half } S\half G T) \neq \vide$,  then
       $$ T^* T  \geq  \tfrac{1}{\lambda}    AT  $$
       for some $A= A\adj \geq 0,$ which emphasizes the strong connection between  the class $\lldeux$ and the reversed  inequality.\\
 
  \hspace*{0.4cm} It is clear from    \eqref{introm3} and \eqref{introm01} that there is no direct relation between $\ldeux$ and $\lldeux.$ However, if $T$ is quasi-similar to $S=S\adj \geq 0$ or, equivalently $T$ and $\tstar$ are quasi-affine to $S$ then one obtains
  $$
\ldeux \ni  T_1  \inclu T \inclu T_2 \in \lldeux.
  $$
 In fact, behind this proof appears the notion of Friedrichs extension of a nonnegative (symmetric operator).
More importantly, when $\rho(T) \neq \vide$ the operators $T$ and $\tstar$ play a symmetric role  with respect to   stronger notions than quasi-similarity, namely   \emph{W}-similarity  and similarity. This can be seen in  Proposition \ref{theosimilarscalar} where the equivalence \eqref{introo1} remains valid even in the unbounded setting. In this case one  obtains the following equivalences:
 $$
T \text{ is $\mathcal{W}$-similar to }  S=S\adj \geq 0 \ifaf T\in \lldeux  \ifaf T\in \ldeux.
  $$
     The assumption $\rho(T) \neq \vide$ is  quite important also for the spectral properties of $T$ (see \cite{paper3}), in particular, if $T \in \ldeux$ such that $\rho(T) \neq \vide$ then
 $$
 \sigma(T ) \inclu \rplus.
 $$

The last part of this paper deals with a particular case, where $T$ is compared to a bounded nonnegative $S \in \lplus.$  Since  both  $\mathcal{W}$-similarity and similarity to such operators imply the boundedness of $T,$ it is enough to restrict attention to quasi-affinity and quasi-similarity notions.


                           \section{The class $\ldeux$ and Sebestyén inequality} \label{section2}


In this section the emphasis will be on  the following  subclass of the closed operators in $\coh :$
   \begin{eqnarray}
     \ldeux=
            \Big\{T=AB \in \coh; \ A \in \lplus \text{ and } B = B\adj \geq 0   \Big\} \label{definitionldeux2},
            \end{eqnarray}
  where $B$ is in general a closed unbounded operator on  $\h.$
   Analogous  to the bounded case, this class is characterized through Sebestyén inequality now involving  unbounded operators. Further extensions are treated in \cref{secrelation}.\\
\hspace*{0.4cm} In the sequel  $T\in \lohk$ stands for    a linear operator from $\h$ to a complex Hilbert space $\kk$ with   domain $\DT$ and range $\ranT$. In addition, one writes $T \in \cohk$ if $T$ is closed.  If $\kk=\h$ then $\coh:=\cohk$ and $\lohk=\loh.$ In this case, $T$ is said to be  \emph{symmetric} if $\la Tx , y\ra = \la x, Ty \ra $ for all $x,y \in \DT.$ If $\la Tx, x \ra \geq 0$ for all $x \in \DT$, then $T$ is \emph{nonnegative}. It is  \emph{selfadjoint} when $\bardt=\h  $ and $\tstar= T$.   Note that  if $T$ is nonnegative and selfadjoint, then it admits a unique  nonnegative selfadjoint square root which will be denoted by $ T\half $; cf. \cite{Sebestyen2017,wouk1966note}.  One writes $T \leq S$ for two  nonnegative selfadjoint operators $S$ and $T$ if
     $$
    \D S\half \inclu \D T\half \qquad \text{and} \qquad
    \| T\half x\| \leq \|S\half x\| \ \text{for all } x \in \D S\half.
    $$
 The class of bounded operators from $\h$ to $\kk$ is denoted by $\bhk$ and  in case $\kk=\h$ this is appropriated to   $\bh.$ If $0 \leq T= \tstar   \in \bh$ then one writes $ T \in \lplus.$

    If T is closed, then its Moore-Penrose inverse   is denoted by  $T^{(-1)}$. It satisfies the following equalities:
    $$
    T T^{(-1)} = P_{\ker {\tstar}\orth} I \rest \ran T \qquad  T^{(-1)} T= P_{\ker T\orth} \rest \DT.
    $$
     The \emph{resolvent set} of $T \in CO(\h)$ is the set $\rho(T)$ of all $\mu \in   \C $ for which $ (T - \mu I)\inv \in B(\h)$.
     The \emph{spectrum} of $T$ is  defined by $\sigma(T)=\C \setminus \rho(T).$

 The next lemma provides a key ingredient for what follows. It treats both densely defined and nondensely defined operators, as well as linear relations; cf. \cref{secrelation}.  Note that its  proof is based on \cite[Lemma 2.9]{hassiderkach2009}, where the equality
\begin{equation}\label{tamer}
  (ST)\adj = \tstar S\adj
\end{equation}
is established in the general  case of linear relations.    
Recall that $(\ref{tamer})$ is satisfied  if $S \in B(\h)$ or $T$ is invertible.

\begin{lemma}\label{relationalpha}
 Let $X \in \lplusk$ and $R $ be a linear relation from  $\h$ to $\kk$, and let $\alpha \in [0,1]$. If $X R\clo$ is closed (closable), then $X^\alpha R\clo$ is closed (closable, respectively) and, moreover,
  \begin{equation}\label{equality}
    (R\adj {X^\alpha})\adj=X^\alpha R\clo.
  \end{equation}
  Analogously, if $\ker X=\{0\}$ and  $R\clo X\inv $ is closed (closable), then $R\clo X^{-\alpha} $ is closed (closable, respectively) and
  \begin{equation} \label{equalityn2}
    (X^{-\alpha} R\adj )\adj=  R\clo X^{-\alpha}.
  \end{equation}
\end{lemma}

\begin{proof}
  Let $(x_n,y_n)   \in  X^\alpha R\clo$ be such that $(x_n,y_n) \converge (x,y) \in \h \times \kk.$ Then, $y_n \in X^{\alpha - 1} XR\clo x_n,$ and therefore
  $$
 X^{1- \alpha}y_n \in X^{1- \alpha}X^{\alpha - 1} XR\clo x_n \inclu XR\clo x_n;
  $$
  here $X^{\alpha - 1}$ denotes a linear relation inverse of $X^{1- \alpha}.$
  Since $X^{1- \alpha} \in \lplusk,$ one has  $X^{1- \alpha}y_n \converge X^{1- \alpha}y$  and   $(x_n,X^{1- \alpha}y_n) \converge (x, X^{1- \alpha}y).$   As $(x_n,X^{1- \alpha}y_n) \in G(XR\clo)$ and   $XR\clo$ is closed, one  concludes that  $(x, X^{1- \alpha}y) \in G(XR\clo).$   On the other hand, $ y \in X^{ \alpha-1} X^{1- \alpha}y,$ which implies that
  $$
  y \in X^{ \alpha-1} (X^{1- \alpha}y)= X^{ \alpha-1} ( XR\clo x) = X^\alpha R\clo x.
  $$
  Consequently,   $X^{ \alpha}R\clo$ is closed.
To prove $(\ref{equality}),$ it suffices to observe  that
   $$
        X^\alpha R\clo=  \left[ \left(X^\alpha R\clo\right)^* \right]^* = \left[ R\adj  (X^\alpha)\adj \right]\adj.
        $$
        \   \hspace*{0.4cm} If $\ker X=\{0\}$ and  $R\clo X\inv $ is closed, then $(R\clo X\inv)\inv= X {R\inv}\clo $ is closed.  Thus, $(\ref{equalityn2})$ follows immediately by applying $(\ref{equality})$ to   $R\inv$ and by taking the inverse. For the closability, it suffices to consider the case where  $(x_n,y_n) \converge (0,y).$
            \end{proof}


     \begin{corollary}
        If $T=AB\in \ldeux$ and $T^{2^n}$  is closed for every $n\in \N$, then
        \begin{equation}\label{eq1induction}
          T^{2^n} = A S_n\in \ldeux \quad \text{for all } n \in \N,
        \end{equation}
        where $(S_n)_{n \in \N}$  is a sequence of nonnegative selfadjoint unbounded operators such that
        $S_0 =B$ and
         $S_{n}= S_{n-1} A  S_{n-1}$ for all   $n \in \N\adj$.
      \end{corollary}
    \begin{proof}
 The case $n=0$ is easily seen.  For $n=1,$ one has $T^2= A (BAB)=A S_1$ and
               \begin{equation}\label{eq2induction}
                   S_1:=BAB= S_0 A S_0= (A\half B)\adj A\half B.
                \end{equation}
 On the other hand $A \in \lplus$ and $AB=T$ is closed, so by Lemma \ref{relationalpha} $A\half B$ is closed. This proves, by $(\ref{eq2induction})$ that $S_1= S_1\adj \geq 0.$ \\
  \hspace*{0.3cm} For $n=2,$ one has
  \begin{align*}
    T^{2^2}= & A [ (B AB) A(B AB) ] = A ( S_1 A S_1)= A S_2,
  \end{align*}
  where
   \begin{equation}\label{eq3induction}
     S_2= S_1 A S_1= (A\half S_1)\adj A\half S_1.
   \end{equation}
     But $ A S_1= AB AB= T^2$ is closed, by hypothesis, $A \in \lplus$ and $S_1$ is closed, so       $A\half S_1$ is closed by Lemma \ref{relationalpha}. Hence, $(\ref{eq3induction}) $ yields that  $S_2= S_2^* \geq 0$. Using again Lemma \ref{relationalpha} and the fact that $T^{2^n}$  is closed, one can conclude by induction that, for all $n \in \N$,  $S_n$ is a nonnegative selfadjoint unbounded  operator such that $S_n= S_{n-1} A S_{n-1}$ and $T^{2^{n}}= A S_n \in \ldeux.$
    \end{proof}

 It is worth mentioning that, in the bounded case, any element $T= AB \in \ldeux$ satisfies the following formula:
\begin{equation}\label{o5t}
  \sigma(A  B) \cup \{0\} = \sigma(BA) \cup \{0\},
\end{equation}
which  easily implies the positivity of the spectrum of $T.$ However, this is a bit more delicate when it comes to the unbounded case. In fact,  $(\ref{o5t})$ is not guaranteed anymore   unless some further spectral properties  are added like    $\rho(AB) \neq \vide$ and $\rho(BA)\neq \vide$; see Hardt et al. \cite{hardt}.  In particular,  for any  unbounded $T \in  \ldeux$ with    $\rho(T)\neq \vide$, it will be shown that     $\sigma(T) \inclu \rplus$.  This motivates the next   results.

\begin{lemma}\label{avanttheospectral}Let $X \in \lplus$ and $T \in CO(\h)$ be a densely defined operator such that  $XT$ is closed.   Then,
      \begin{equation}\label{ma2oul01}
          (X\half T\adj  X\half  )\adj= X\half T X\half.
     \end{equation}
     Moreover,      if  $T=\tstar$ and   $\rho(XT) \neq \vide$, then
    \begin{equation}\label{ma2ouul2}
      \sigma(XT)  = \sigma(X\half T X\half)  \inclu \R,
    \end{equation}
   in particular,  $0 \in \rho(XT) \ifaf 0 \in \rho(X\half T X\half).$
\end{lemma}

 \begin{proof}
 Observe that
  \begin{equation}\label{kara1}
 ( X\half \tstar X\half)\adj=( X\half (X\half T)\adj  )\adj  =  (X\half T)\clo  (X\half)\adj  = \overline{X\half T} X\half.
 \end{equation}
 Since $XT$ is closed, it follows from    Lemma \ref{relationalpha} that $X\half T$ is closed. This yields by $(\ref{kara1})$ that $
 ( X\half \tstar X\half)\adj =  X\half T X\half.$ \\
 \hspace*{0.4cm} Assume now that $\rho(XT) = \rho \left( X\half (X\half T )\right) \neq \vide$ and  $\tstar=T$.  Then, $(\ref{kara1})$ shows that $X\half T X\half$ is selfadjoint, and hence  $\rho(X\half T X\half) = \rho\left( X\half (T X\half )\right)  \neq \vide.$  Using \cite[Lemma 2.2]{hardt}  and  \cite[Lemma 2.4]{hardt}, one then concludes that
           \begin{equation}\label{eq1tunis}
          \sigma(XT) \cup \{0\}= \sigma\left( X\half (X\half T) \right) \cup \{0\} = \sigma\left( (X\half T) X\half  \right) \cup \{0\} \inclu \R .
           \end{equation}
  Now assume that   $0 \in \rho(XT)$. Then $\ran XT=\h= \ran X\half$, and hence $X\half$ is invertible. This implies the invertibility of $T$, so   $0 \in \rho(X\half T X\half).$ Similarly, the invertibility of $X\half T X\half$ ensures  that of $T,$ which proves the remaining implication.  Together with  $(\ref{eq1tunis})$, this shows $(\ref{ma2ouul2}).$
 \end{proof}

Thanks to the previous lemma, it will be seen in Proposition  \ref{theospectral} how any element of $\ldeux$ is connected to a nonnegative selfadjoint operator. This connection is introduced in the following definition and it will be further developed in \cref{section04}.
 \begin{definition}\normalfont
   Let $T,S \in LO(\h).$ If there exists   $G \in B(\h)$ such that   $TG=GS$ then    $T$ is said to be \emph{pre-similar} to   $S$ with interwining operator $G.$
 \end{definition}

\begin{proposition}\label{theospectral}
    If $ T=AB \in \ldeux$ then
   $
          (A\half B  A\half  )\adj=   A\half B  A\half
    $
 and $T$ is pre-similar to $A\half B  A\half $ with interwining operator $A\half.$      Moreover, if   $   \rho(T) \neq \vide$ ,    then
    \begin{equation*}
      \sigma(T)  = \sigma(A\half B  A\half)  \inclu \rplus.
    \end{equation*}
 \end{proposition}

\begin{proof}
Since by definition $A  \in \lplus$ and $AB$ is closed,  it follows from  Lemma \ref{avanttheospectral}  that
    $S:= A\half B A\half $ is a nonnegative selfadjoint operator such that
    $T A\half= A\half (A\half B A\half ) = A\half S.$ Hence,
    $T$ is pre-similar to $S.$ The remaining result follows immediately again from Lemma \ref{avanttheospectral}.
\end{proof}

\subsection{Sebestyén inequality}

In this section,  Sebestyén's theorem is generalized to the case of unbounded operators. The case of bounded operators was originally proved in \cite{sebestyen1983restrictions}, for a recent treatment see also \cite{gustavo2013products,product2021}, where  the following equivalence is stated for $T,B  \in \bh:$
\begin{equation}\label{mot1}
\tstar T \leq \lambda \tstar B, \ \lambda \geq 0 \quad  \ifaf \quad T= XB  \quad \text{for some } X \in \lplus.
\end{equation}
 The following lemma serves as a first step towards the generalization of \eqref{mot1} and is a useful tool for some further results. The equivalence of (i) and (ii) holds  even in the case of linear relations; cf.  \cite[Lemma 4.2]{hassi22014},  and for related results see also  \cite[ Theorem 2.2]{popovici2013factorizations},
  where $T \inclu BY \ifaf \ran T \inclu \ran B  $ and     \cite[Lemma 3.1]{sandovici2025}, where  $YB \inclu T \ifaf \ker B \inclu \ker T,$  respectively are established for   linear relations $T,B$ and $Y$.

\begin{lemma} \label{lemprepthe}
Let $T,B:\h \rightarrow \kk $ be closed densely defined linear operators.  Then the following statements  are equivalent:
   \begin{enumerate}\def\labelenumi{\rm(\roman{enumi})}
     \item $Y B \inclu T$ has a solution $Y \in \bk;$
   \item   $T^* T  \leq c^2 B\adj B$  for some $0 \leq c \ (  =\n Y\n).$
      \end{enumerate}
     In this case, $Y$ can be selected such that $\ran Y \inclu \ranbarT$ and $\ker B\adj \inclu \ker Y.$ Furthermore,  if $\tstar B$ is selfadjoint then the following implication holds
      \begin{equation}\label{darkness1}
       Y B \inclu T  \text{ for some } Y \in \lplusk \Rightarrow T^* T  \leq  c_1  \tstar B \leq c_2 \, B\adj B,
      \end{equation}
       where $c_1,c_2 \geq 0.$
     In this case $\DB \inclu \D \halftsb \inclu  \DT$ and
       \begin{equation}\label{darkness2}
       \tstar B= B\adj Y B = B\adj T.
      \end{equation}
\end{lemma}

\begin{proof} The implication $ \rm (i) \Rightarrow (ii)$ is clear since $\n YB x \n \leq \n Y\n \n B x\n$ for all $x \in \D B.$ To see the reverse implication notice that $G B x = Tx ,$ $x \in \D B$ is a well-defined operator with $\n G \n \leq c.$ Then, $Y \in B(\kk)$ is obtained by continuation of $G$ to $\ranbar B$ and using the zero extension to $(\ran B )\orth= \ker B\adj$, so that $\ker B\adj \inclu \ker Y.$  \\
         \hspace*{0.4cm} Now, assume that $\tstar B$ is selfadjoint and $YB \inclu T$ for some $Y \in \lplusk.$
          Then, $ Y\half  \overline{Y\half B} \inclu \overline{Y\half  \overline{Y\half B}} = \overline{Y  B} \inclu T$ and the first part of the lemma shows that there exists $0 \leq c_1 \leq \n Y\half\n$  such that
       \begin{equation}\label{hahahahah0a}
       \tstar T \leq   {c_1}^2 \, ( Y\half B)\adj \overline{Y\half B}.
        \end{equation}
        On the other hand, one has
        $$
        \tstar B \inclu    (YB)^*B=
          (Y\half Y\half B)^*B=( Y\half B)^*Y\half B \inclu
        ( Y\half B)^*\overline{Y\half B}.
         $$
 Since $  \tstar B$ is selfadjoint, it follows  that
         \begin{equation}
            \tstar B = B\adj Y B= ( Y\half B)^*\overline{Y\half B},
         \end{equation}
       which shows the first identity in  \eqref{darkness2}. Moreover, one has
         $$
         B\adj T \inclu \tstar B=( Y\half B)^*\overline{Y\half B}= B\adj (Y B) \inclu B\adj T,
         $$
        which  means that
               \begin{equation*}
              B\adj T =\tstar B=  B\adj Y B =   ( Y\half B)^*\overline{Y\half B} \leq \n Y \n B\adj B.
         \end{equation*}
       Combining this with   \eqref{hahahahah0a} leads to
        $$
        \tstar T \leq  {c_1}^2 \, \tstar B \leq {c_1}^2 \n Y\n B\adj B,
        $$
        which completes the proof of \eqref{darkness1}, \eqref{darkness2}   and $\DB   \inclu \D \halftsb \inclu \D T.$
\end{proof}

Motivated by Lemma \ref{lemprepthe}, the next step towards the extension of    the equivalence  \eqref{mot1}   is to address  the implication in  the following equivalence:
     \begin{equation}\label{mot2}
        \tstar T \leq \lambda \tstar B, \ \lambda \geq 0 \quad  \ifaf \quad  XB \inclu T \quad \text{for some } X \in \lplus.
      \end{equation}
For this, begin by observing that      in the  general case of closed densely defined operators, and for $B_0:=B \rest \D \tstar B$, one has $\tstar B= \tstar B_0 \inclu \tstar \overline{B_0} \inclu \tstar B$. This means that
\begin{equation}\label{motivation}
\tstar B= \tstar B_0 = \tstar \overline{B_0} = \tstar B,
\end{equation}
so  the following equivalence holds for $\lambda \geq 0$
\begin{equation}\label{motivation2}
\tstar T \leq \lambda \tstar B \ \ifaf \tstar T \leq \lambda \tstar \overline{B_0}.
\end{equation}
However, contrary to the bounded case where automatically $B_0= \overline{B_0}=B$, one cannot expect the factorization $T=XB$ as in $\eqref{mot1}$  since one only has
\begin{equation*}
  \overline{B_0}   \inclu B.
\end{equation*}
Thus, it becomes natural to restrict $B$ to $\overline{B_0}$ in the following extension of Sebestyén theorem.
\begin{theorem} \label{firsttheo01}
  Let $T,B:\h \rightarrow \kk $ be closed densely defined linear operators such that $\tstar B $ is a selfadjoint operator and let $B_0= B\rest\D \tstar B.$ Then the following assertions are equivalent for some $0 \leq \lambda \ (  =\n X\n):$
   \begin{enumerate}\def\labelenumi{\rm(\roman{enumi})}
   \item   $T^* T  \leq \lambda \tstar B$;
     \item $X\overline{B_0} \inclu T$ has a solution $X \in \lplusk$.
      \end{enumerate}
      In this case
      \begin{equation}\label{work}
       (B_0)\adj X \overline{B_0} =  \tstar \overline{B_0}=(B_0)\adj T
      \end{equation}
      and, moreover,       $X$ can be chosen such  that $\kerr \tstar \inclu \ker X$ with $\n X \n \leq \lambda.$ In particular,
      \begin{equation}\label{secondseb}
        T^* T  \leq \lambda \tstar \overline{B_0} \text{ and }  \DT \inclu \D \overline{B_0} \ifaf T=X \overline{B_0} \text{ for some } X \in \lplusk.
      \end{equation}
      In this case   $\ker X = \ker \tstar.$
\end{theorem}

\begin{proof}
 Assume (i). Then  a direct application of Lemma \ref{lemprepthe} to $T$ and $\halftsb$   leads to the existence of  $G_0 \in \bhk$ such that $\ranbar G_0 \inclu \ranbarT$,  $ \ker \tstar B \inclu \ker G_0 $ and
     \begin{equation}\label{0sfx001}
          G_0 \halfltsb \inclu T.
      \end{equation}
     Hence
    \begin{equation}\label{sfx001}
    \tstar \inclu   \halfltsb (G_0)\adj
     \end{equation}
  and
      \begin{equation}\label{sfx002}
 \lambda    \tstar B \inclu \halfltsb \lambda  (G_0)\adj B.
     \end{equation}
Multiplying  $\eqref{sfx002}$   from the    left by $(\lambda \tstar B)^{(-\frac{1}{2})}$, one obtains
      \begin{equation}\label{choftek001}
                  \ppa I_{\D \halftsb} \halfltsb  \inclu  \ppa I_{\D \halftsb} \lambda (G_0)\adj B \inclu  \lambda (G_0)\adj B.
                   \end{equation}
This implies that
\[               
 (\ppa I_{\D \halftsb} \halfltsb) \rest \D \tstar B =  \lambda (G_0)\adj B \rest \D \tstar B= \lambda (G_0)\adj B_0,
\]                              
and hence
\begin{equation}\label{chotek002}
   \halfltsb \rest \D \tstar B=\lambda (G_0)\adj B_0.
\end{equation}
Since $\D \tstar B $ is a core for $\halftsb,$ i.e. $\halftsb =  \overline{\halftsb \rest \D \tstar B }$,  one concludes that
\begin{equation}\label{celine}
            \halfltsb=  \overline{\halfltsb \rest \D \tstar B}  =  \overline{\lambda (G_0)\adj B_0} \supseteq  \lambda (G_0)\adj \overline{B_0}.
\end{equation}
Together with $\eqref{0sfx001}$ this implies that
                $
             \lambda  G_0 (G_0)\adj \overline{B_0} \inclu T
                 $
                and therefore
                $$
                X \overline{B_0} \inclu T
                $$
                 with $ X=\lambda G_0 (G_0)\adj \in \lplusk$ so that  $\n X \n \leq \lambda.$

    \hspace*{0.4cm} The reverse implication as well as the equalities in \eqref{work} follow immediately from Lemma \ref{lemprepthe}. \\
  \hspace*{0.4cm} The inclusion $\kerr T\adj \inclu \kerr X$ follows easily from the construction of $G_0,$  the identity  $\eqref{sfx001}$ and from the fact that  $\ker (G_0)\adj= \ker X.$

   Now assume  that $\DT \inclu \D \overline{B_0}$ and $  T^* T  \leq \lambda \tstar \overline{B_0}.$  Then,  the implication $\rm "(i)  \Rightarrow (ii)"$ immediately yields that $X\overline{B_0}= T$ for some $X \in \lplusk.$ For the converse, observe that $X \overline{B_0}= T$ is closed, so   $X\half \overline{B_0}$ is closed by  Lemma  \ref{relationalpha}. Consequently,
    $$
    \tstar \overline{B_0} = (B_0)\adj X \half X\half \overline{B_0}= (  X\half \overline{B_0})\adj  X\half \overline{B_0}$$
    is a nonnegative selfadjoint operator with $\D ( \tstar \overline{B_0})\half = \D X\half \overline{B_0} = \D \overline{B_0}= \DT.$ Moreover,
    $
    \tstar T = \tstar  X \overline{B_0} \leq  \n X \n \tstar \overline{B_0},$  which completes the argument. On the other hand, one has $\tstar = {B_0}\adj X$, so $\ker X \inclu \ker \tstar.$ Consequently, $\ker \tstar = \ker X$ by the first part of the proof.
\end{proof}

 \begin{remark}\label{remark1}\normalfont
   \begin{enumerate}\def\labelenumi{\rm(\roman{enumi})}
\item The inequality in item (i) of Theorem \ref{firsttheo01} induces  the following new inequality
\begin{equation}\label{admitrem}
 \tstar \overline{B_0} \leq \mu (B_0)\adj   \overline{B_0},
\end{equation}
where $\mu = \n X \n. $  This follows from Lemma \ref{lemprepthe}, \eqref{darkness1}. Notice that the inclusion $\overline{B_0} \inclu B$ implies that   $B\adj B \leq  {B_0}\adj \overline{B_0},$ and hence \eqref{admitrem} does not necessarily imply  the inequality $\tstar B \leq \gamma B\adj B,$ $\gamma \geq 0.$
\item
The inequality  \eqref{admitrem} is not sufficient to prove item (i) of Theorem \ref{firsttheo01}. However, one can always obtain the following equivalence
 \begin{align} \label{3equiv}
\begin{split}
  T^* T  \leq \lambda  \tstar   & \overline{B_0}  \leq \lambda  \, \mu (B_0)\adj   \overline{B_0} \\
    & \Updownarrow \\
        X\overline{B_0} \inclu  T \text{ has} & \text{ a solution } X \in \lplusk.
\end{split}
\end{align}
 \item  By construction,  $\lambda = 0$ if and only if  the solution $X=0,$ in which case $T=0.$
\end{enumerate}
 \end{remark}

Although Theorem \ref{firsttheo01} establishes the equivalence \eqref{mot2} only for $B_0$, its proof reveals that an additional condition would allow the desired equivalence to hold for $B,$  more generally. This can be seen in the following remark.

\begin{remark}\label{remarko5ra}
\normalfont
  Following Remark \ref{remark1}, a particular case of Theorem \ref{firsttheo01} where  $\D \tstar B$ is a core for $B$ leads to the following statements for $\lambda \geq 0:$
\begin{enumerate}
  \item $ X B \inclu T $   for some $ X \in \lplusk \ifaf  T^* T  \leq \lambda \tstar B.$
  \item $T=X B$ for some $ X \in \lplusk \ifaf  T^* T  \leq \lambda \tstar B$ and $ \DT \inclu \D B.$
\end{enumerate}
\end{remark}

In the absence of the additional core conditions stated in Remark \ref{remarko5ra}, the question arises about the most appropriate generalization of \eqref{mot1} to the unbounded setting. Motivated by  \eqref{3equiv}, this question naturally leads to consider whether the converse of \eqref{darkness1} in  Lemma \ref{lemprepthe} is true. Since the latter implies that
$$
\DB \inclu \D \halftsb  \text{ and }   B\adj T = \tstar B,
$$
it becomes natural   also to impose these conditions in the following result, which in fact constitutes the final step towards the objective  of this subsection.

\begin{proposition} \label{sebestyenB}
   Let $T,B:\h \rightarrow \kk $ be closed densely defined linear operators such that $\tstar B=B\adj T $ is selfadjoint.  Then the following assertions are equivalent for some $0 \leq \lambda \ (  =\n X\n):$
   \begin{enumerate}\def\labelenumi{\rm(\roman{enumi})}
   \item   $T^* T  \leq \lambda \tstar B$  and $\DB \inclu \D \halftsb =\DT $;
   \item $ X B \inclu T$ has a solution $X \in \lplusk$;
     \item $\overline{X B} =T$ has a solution $X \in \lplusk.$
      \end{enumerate}
      In this case
     \begin{equation}\label{workfin2}
       B\adj X B =  \tstar B=B\adj T
      \end{equation}
      and, moreover, $X$ can be chosen such that $\kerr \tstar =\ker X$ with $\n X \n \leq \lambda.$
\end{proposition}

\begin{proof} Assume (i). Then, following the same reasoning as   in the proof of Theorem \ref{firsttheo01}, \eqref{0sfx001} together with the fact that $\DT = \D \halftsb$ gives
      \begin{equation}\label{eqkdhm1}
       T= G_0 \halfltsb,
       \end{equation}
   and hence
    \begin{equation}\label{eqkdhm}
      B\adj T= B\adj G_0 \halfltsb.
    \end{equation}
  As $B\adj T = \tstar B $ is nonnegative and selfadjoint, multiplying \eqref{eqkdhm} from the right by $( B\adj T )^{(-\frac{1}{2})}$ implies that
  \begin{eqnarray}
  \halfbt \left( P_{{\ker B\adj T}\orth} \rest  \halfbt( \D B\adj T)   \right)    \inclu B\adj G_0 \label{dense} .
  \end{eqnarray}
 Since  $\D \tstar B$ is a core for $\halfbt,$ the set $P_{{\ker B\adj T}\orth} \rest  \halfbt( \D B\adj T)$ is dense in $\h$  and, therefore, $\dombar B\adj G_0 =\h$ by $\eqref{dense}.$  Hence,  $ {G_0}\adj B $ is a closable operator which satisfies
  $$
  \halfltsb = \overline{\lambda {G_0}\adj B_0} \inclu  \overline{\lambda {G_0}\adj B};
    $$
    see  \eqref{celine}.
    Therefore, $\DB \inclu \D \halftsb$ implies that
    $$
   \lambda {G_0}\adj B= \overline{\lambda {G_0}\adj B} \rest  \DB  \inclu \overline{\lambda {G_0}\adj B} \rest  \D \halftsb   =  \halfltsb  \inclu \overline{\lambda {G_0}\adj B}.
    $$
    Consequently  $ \overline{\lambda {G_0}\adj B} =   \halfltsb,$ which by \eqref{eqkdhm1} gives
    $$
     T= \lambda G_0   \overline{ {G_0}\adj B} = \lambda \overline{ G_0   \overline{ {G_0}\adj B} } = \lambda  \overline{ G_0     {G_0}\adj B } \supseteq G_0     {G_0}\adj B.
     $$
        This completes the proof of the implication $ \rm (i) \Rightarrow (iii) \Rightarrow (ii)$ for $X = G_0 (G_0)\adj \in \lplusk.$ The implication  $ \rm (ii) \Rightarrow (i)$ together with the identity \eqref{workfin2} is immediate from Lemma \ref{lemprepthe}.\\
        \hspace*{0.4cm} To see that $\ker \tstar = \ker (G_0)\adj,$ observe from   \eqref{eqkdhm1} that $\tstar=  \halfltsb (G_0)\adj$ and hence $ \ker X= \ker(G_0)\adj \inclu \ker \tstar.$  On the other hand,   the inclusion $XB \inclu T$ together with Lemma \ref{lemprepthe} shows that $\ranbar X \inclu \ranbarT$ or, equivalently, $\ker \tstar \inclu \ker X.$   Consequently, $\ker X=\ker \tstar.$
\end{proof}

Observe that, under the assumptions of Proposition \ref{sebestyenB}, items $\rm (i) - (iii)$ are equivalent to the following statement for some $\lambda \geq 0$:
      \begin{equation}\label{further}
         T^* T  \leq \lambda \tstar B \leq \lambda^2 B\adj  B  \quad \text{   and} \quad
       \D \halftsb =\DT .
       \end{equation}
   Moreover, some further   necessary and  sufficient conditions for \eqref{further}  may be derived  through the study of forms, as investigated in  \cite{Hassi25sebestyen}.
\begin{corollary}\label{khsr}
     Let $T,B:\h \rightarrow \kk $ be closed densely defined linear operators. Then the following assertions are equivalent for some $0 \leq \lambda \ (  =\n X\n):$
   \begin{enumerate}\def\labelenumi{\rm(\roman{enumi})}
   \item  $\tstar B$ is a selfadjoint operator such that $\DT \inclu \DB \inclu  \D \halftsb$ and  $T^* T  \leq \lambda \tstar B =\lambda B\adj T;$
   \item $ X B =T$ has a solution $X \in \lplusk.$
      \end{enumerate}
\end{corollary}
\begin{proof}
 If (i) holds then $\D \halftsb \inclu \DT,$ and hence item (ii) easily follows from the implication $\rm(i) \Rightarrow (ii)$ of Proposition \ref{sebestyenB}. Conversely, if $T=XB$ has a solution $X \in \lplus$ then $X\half B$ is closed, by Lemma \ref{relationalpha} and therefore
  $
  \tstar B= B\adj  X B=B\adj T=  (X\half B)\adj X\half B
  $
  is a nonnegative selfadjoint operator. Hence, one concludes the result from Lemma \ref{lemprepthe}, \eqref{darkness1} and from the fact that $\DT= \DB.$
\end{proof}
A consequence of Corollary \ref{khsr} leads to the characterization of the class $\ldeux$ by means of Sebestyén inequality described in the following theorem, thereby generalizing \cite[Theorem 4.5]{product2021}.
\begin{theorem}\label{sebldeux}
  Let $T \in \coh$ be a densely defined operator. Then, $T \in \ldeux$ if and only if  $\tstar T \leq    \, \tstar Y = YT$ admits a solution $Y=Y\adj \geq 0$ such that $\tstar Y$ is selfadjoint and  $\DT \inclu  \D Y  \inclu \D( \tstar Y)\half.$
\end{theorem}

\begin{proof}
The proof follows immediately by applying Corollary \ref{khsr} to $B=Y.$
\end{proof}

          \subsection{$\ldeux$ and quasi-affinity   to $S= S\adj \geq 0$ }\label{subsection1}

 In this subsection, for the convenience of the reader,  \emph{\text{$G$}-quasi-affinity} refers to \emph{quasi-affinity}  already mentioned in the introduction.
 The following lemma provides a  link between the $G$-quasi-affinity and the $|G|$-quasi-affinity to a nonnegative selfadjoint operator, which will be useful in Subsection \ref{subseection43}.

\begin{lemma}\label{remquasiaffini}
Let $T \in LO(\h)$ be a densely defined operator. Then  the following assertions are equivalent:
\begin{enumerate}[(i)]\def\labelenumi{\rm(\roman{enumi})}
  \item  $\overline{GT G\inv}= {G\inv}\adj \tstar G\adj \geq 0$ for a quasi-affinity $G \in \bh;$
  \item $ \overline{ X\half T X\mhalf} = X\mhalf \tstar X\half \geq 0$ is selfadjoint for a quasi-affinity    $X  \in \lplus$.
\end{enumerate}
\end{lemma}

%
\begin{proof}
  Assume \rm{(i)} and   let $G= U |G|$ be the polar decomposition of $G.$ Since $G$ is a quasi-affinity, $U$ is unitary. Setting  $X:=G\adj G,$ one   sees that $X \in \lplus$ is a quasi-affinity and
  \begin{equation}\label{etoileuh}
    {G\inv}\adj \tstar G\adj  = U |G|\inv \tstar |G| U\adj = U (    X\mhalf \tstar X\half ) U\inv.
  \end{equation}
 As $U$ is unitary, one concludes from \eqref{etoileuh} and  $(i)$ that  $X\mhalf \tstar X\half \geq 0$ is   selfadjoint and hence $X\mhalf \tstar X\half = (X\mhalf \tstar X\half)\adj= \overline{X\mhalf \tstar X\half} \geq 0$. The reverse implication is immediate.
\end{proof}

The following theorem  establishes a connection between the class $\ldeux$ and the quasi-affinity to nonnegative selfadjoint operators.


\begin{proposition}\label{inclusionnfs}
  Let $T\in CO(\h) $ be densely defined. Then the following assertions are equivalent:
   \begin{enumerate}[(i)]\def\labelenumi{\rm(\roman{enumi})}
       \item  $\tstar$ is  quasi-affine to   some  $S = S\adj \geq 0;$
       \item $T \supseteq  AB \in \ldeux$  with $\ranbarA =\h;$
        \item there exists a quasi-affinity $X \in \lplus$   such that
        $$
         0 \leq  \tstar X\inv \inclu  X\inv  T \ifaf  0 \leq X\tstar \inclu TX.
         $$
   \end{enumerate}
\end{proposition}

\begin{proof}
$\mathrm{(i)} \Rightarrow \mathrm{(ii)} $ Assume that  $\tstar$ is $G$-quasi-affine to      $S =S\adj \geq 0.$  Then the inclusion $G \tstar \inclu S G$ implies that
\begin{equation}\label{retard}
  \tstar (G\adj G)\inv  \inclu G\inv S (G\inv)\adj := B_0 \geq 0
\end{equation}
 and hence $\dombar B_0 = \ranbar G\adj G = \h.$    Now, let $B_F$  be the Friedrichs extension of $B_0$ (cf. \cite{kato1980}) and let $A= G\adj G \in \lplus.$  Then  \eqref{retard} shows that    $\tstar A\inv  \inclu   B_0 \inclu B_F$, and therefore $A B_F \inclu T$.   \\
      \hspace*{0.4cm}  $\mathrm{(ii)} \Rightarrow \mathrm{(iii)} $ Since $AB \inclu T \in \ldeux$ and $\ranbarA= \h,$ it follows that $\ker A= \{0\}$ and one has  $0 \leq B \inclu A\inv T.$ Hence,
       \begin{equation}\label{trin}
         0 \leq \tstar A\inv \inclu  (A\inv T)\adj \inclu B \inclu A\inv T = (\tstar A\inv)\adj.
       \end{equation}
       By taking $X= A,$ one concludes that $0 \leq \tstar X\inv \inclu   X\inv T $ or, equivalently, $ X \tstar   \inclu   T X .$ Moreover, it follows from \eqref{trin} that
       $ X\tstar \inclu XBX \geq 0,$ which completes the proof of (iii).  \\
  \hspace*{0.4cm}  $\mathrm{(iii)} \Rightarrow \mathrm{(i)} $ Since $\ranbar X= \h=\dombar \tstar$ it follows that $ \tstar X\inv \geq 0$ is a densely defined operator whose Friedrichs extension is again denoted by $B_F.$     Then
  $
   \tstar X\inv \inclu B_F
   $
   and   one has
   $$
   X\half \tstar  \inclu  (X\half B_F X\half) X\half  \inclu \left( \, \overline{X\half {B_F}\half} (X\half {B_F}\half)\adj \right) X\half.
   $$
   This proves that $\tstar$ is $X\half$-quasi-affine to $ S:=  \overline{X\half {B_F}\half} (X\half {B_F}\half)\adj  \geq 0.$
        \end{proof}

\begin{remark}\normalfont \label{rembf} In the proof of Proposition \ref{inclusionnfs} the operator $  B_0 $ in \eqref{retard} is nonnegative and densely defined. Hence the form generated by $B_0$ is closable and its closure has $B_F$, the Friedrichs extension, as the unique representing   operator given by
  \begin{equation}\label{BF}
    B_F = (G\inv S\half) \overline{S\half  {(G\inv)}\adj},
  \end{equation}
cf. \cite{kato1980}.
The proof also shows that if $B$ is any nonnegative selfadjoint extension of $B_0$ then (ii) holds and (iii) follows by taking $X=A.$
\end{remark}

The rest of this section is devoted to describe close relations  between Sebestyén inequality and  quasi-affinity to a nonnegative selfadjoint operator.

 \begin{corollary}\label{remsebsqua}
Let $T \in CO(\h)$ be a densely defined operator and let   $S=S\adj \geq 0.$    If  $\tstar$ is $G$-quasi-affine to $S$  such that $\rho ( \tstar B_F) \neq \vide$, then  there exists $\lambda >0$ such that
  $$\tstar T \leq \lambda \tstar B_F,
  $$
where $B_F$ is defined in \eqref{BF}.
\end{corollary}

\begin{proof}
  Since $\tstar$ is $G$-quasi-affine to some $S=S\adj \geq 0,$ it follows from Proposition \ref{inclusionnfs} and Remark \ref{rembf} that $AB_F \inclu T$ with $B_F= (G\inv S\half) \overline{S\half  (G\inv)\adj}.$ Hence
  $$
  \tstar B_F \inclu  B_F A B_F= (A\half B_F)\adj A\half B_F  \inclu (A\half B_F)\adj \overline{A\half B_F}=:M \geq 0.
    $$
    Since $M$ is selfadjoint, $\tstar B_F$ is symmetric. On the other hand, $\rho( \tstar B_F) \neq \vide$ by assumption and therefore $\tstar B_F$ is selfadjoint, too.
    Together with the fact that $A  \overline{B_F\rest \D \tstar B_F}  \inclu A B_F \inclu  T $ this yields $T^*T \leq \lambda\, T^*B$ for some   $\lambda \geq 0$ by  Theorem \ref{firsttheo01}.
\end{proof}

Note that a small adjustment to item (i) of Proposition  \ref{inclusionnfs} allows
$T$  to be written as the product of two nonnegative, in general, unbounded  linear operators  motivating the following result.
\begin{proposition} \label{egalquasiaff}
  Let $T\in CO(\h) $ be densely defined. Then the following assertions are equivalent:
   \begin{enumerate}[(i)]\def\labelenumi{\rm(\roman{enumi})}
       \item  $\tstar$ is $G$-quasi-affine to  $S = S\adj \geq 0$ with $\D T \inclu  \D G\inv S\half \overline{S\half  (G\inv)\adj};$
       \item $T =  AB \in \ldeux$  with $\ranbarA =\h;$
        \item there exists a quasi-affinity $X \in \lplus$  such that
        $$
          X\inv T = \overline{\tstar X\inv}  \geq 0,
                   $$
                   where $ \D  X\inv T=\DT.$
   \end{enumerate}
\end{proposition}
 \begin{proof}$\rm (i) \Rightarrow (ii)$
 Using the same arguments as in the proof of Proposition \ref{inclusionnfs}, one obtains $A B_F \inclu T.$
 On the other hand, $B_F = G\inv S\half \overline{S\half  (G\inv)\adj}$ by  Remark~\ref{rembf}  and hence the assumption   $\DT \inclu  \D B_F  $  yields $T = A B_F \in \ldeux.$ \\
      \hspace*{0.4cm}  $\mathrm{(ii)} \Rightarrow \mathrm{(iii)} $ For $X= A$  one has $ X\inv T = B = B\adj = (X\inv T)\adj = \overline{ \tstar X\inv} \geq 0.$ Moreover, $\DT = \DB = \D X\inv T.$\\
      \hspace*{0.4cm}  $\mathrm{(iii)} \Rightarrow \mathrm{(i)}$ Set $B= X\inv T.$ Then $XB =X X\inv T \inclu T$ and, since $\DT = \D B$ it follows that $T =XB \in CO(\h).$ Thus $X \half B$ is a  closed densely defined operator  by Lemma \ref{relationalpha} and hence
      \begin{equation}\label{piqrr}
         S:= X\mhalf T X\half = X\half B X\half = X\half B\half (X\half B\half)\adj  \geq 0,
      \end{equation}
      is a selfadjoint operator. Moreover, it follows from \eqref{piqrr} that
      $$
      S =  (X\mhalf T X\half)\adj \supseteq X\half \tstar X\mhalf,
      $$
      and therefore $X\half \tstar  \inclu S X\half.$ This proves that $\tstar$ is $X\half$-quasi-affine to $S.$ On the other hand, multiplying \eqref{piqrr} from the left by $X\mhalf$ and from the right by $X\half$ shows that
      $
      X\mhalf S X\half = BX = \tstar.
      $
       Thus $\tstar X\inv  \inclu X\mhalf S X\mhalf \geq 0,$ which implies  that $B_0:= X\mhalf S X\mhalf$ is a densely defined operator such that
       $$
       \overline{\tstar X\inv}  \inclu \overline{B_0} \inclu {B_0}\adj \inclu  \overline{\tstar X\inv}.
       $$
       Consequently, $ \overline{\tstar X\inv}  = \overline{B_0} = B_F,$ where $B_F = X\mhalf S\half  \overline{S\half  X\mhalf}  $ is the Friedrichs extension of $B_0$ and $\D B_F = \D \overline{\tstar X\inv} = \D X\inv T = \DT.$
 \end{proof}

 The reversed implication for Corollary \ref{remsebsqua} is established in the next result where a subclass of $\ldeux$ is characterized not only by Sebestyén  inequality but also by quasi-affinity to a nonnegative selfadjoint operator.

\begin{theorem}\label{coralama}
Let $T \in \coh$ be a densely defined operator with $\ranbarT =\h$ and let   $S=S\adj \geq 0.$ Then the following statements are equivalent:
\begin{enumerate}[(i)]\def\labelenumi{\rm(\roman{enumi})}
  \item $\tstar  T \leq \lambda \tstar B  $ with $  \D B \inclu  \DT \inclu \D \overline{B \rest \tstar B}$  for some $ \lambda \geq 0$ and $ B= B\adj \geq 0;$
  \item $T = A\, \overline{B \rest \tstar B} \in \ldeux$ with $ \ranbarA = \h;$
  \item $\tstar \text{ is } G\text{-quasi-affine to } S = S^* \geq 0  \text{ with }  \DT \inclu \D \overline{B_F \rest \D(\tstar B_F}),$ where $B_F=G\inv S\half \,  \overline{S\half  {(G\inv)}\adj}.$
\end{enumerate}
\end{theorem}

\begin{proof}
 $\mathrm{(i)} \ifaf \mathrm{(ii)} $
   Observe from \eqref{secondseb} in Theorem \ref{firsttheo01}   that $T = A\overline{B_0} \inclu AB,$ where $A \in \lplus$ and  $B_0= B \rest \D \tstar B.$ Since $\DB \inclu \DT$ it follows that $T=A\overline{B_0} = A B,$ and hence $\DB =  \D\overline{B_0}.$ This implies  that $\overline{B_0}= B = B\adj \geq 0,$ and hence $T \in \ldeux$ with   $\ranbar T = \ranbar A =   \h.$      The reverse implication follows  immediately  from Theorem \ref{firsttheo01} by choosing $B= \overline{B \rest \D \tstar B}.$  \\
  \hspace*{0.3cm} $\mathrm{(ii)} \ifaf \mathrm{(iii)} $
   Assume (ii). Then, it is clear  from Proposition \ref{egalquasiaff} that $\tstar $ is  quasi-affine to some $ S = S^* \geq 0$ with $ \DT \inclu \D B_F.$ Moreover, the proof of Proposition \ref{egalquasiaff} shows that $B_F =\overline{B \rest \D \tstar B} ,$ which completes the argument of the direct implication. To see the reverse implication, observe from the Proposition \ref{egalquasiaff} that  $T =  AB_F \in \ldeux$  with $\ranbarA =\h.$ Hence $\DT = \DB_F \inclu \D \overline{B_F \rest \D(\tstar B_F})$, and thus  $ B_F= \overline{B_F \rest \D(\tstar B_F})$ satisfies (ii).
\end{proof}

The next remark contains a variant of Theorem \ref{coralama} and  gives necessary and sufficient conditions for an operator $T$ with $\ranbarT =\h$  to be in $\ldeux.$

\begin{remark}
Let $T \in \coh$ be a densely defined operator with $\ranbarT =\h$  and let   $S=S\adj \geq 0.$ Then the following statements are equivalent:
\begin{enumerate}[(i)]\def\labelenumi{\rm(\roman{enumi})}
  \item $\tstar  T \leq \lambda_0 \tstar B \leq \lambda_1 B\adj  B  $ for some $ \lambda_0, \lambda_1 \geq 0$   with $  \DT \inclu \DB$, $\tstar B= B\adj T$     and $ B= B\adj \geq 0;$
  \item $T = AB \in \ldeux;$
  \item $\tstar \text{ is } G\text{-quasi-affine to } S = S^* \geq 0  \text{ with }  \DT \inclu \D B_F ,$ where $B_F=G\inv S\half \,  \overline{S\half  {(G\inv)}\adj}.$
\end{enumerate}
\end{remark}

Note that   once Corollary \ref{coralama} or Corollary \ref{remsebsqua}    is applied to $T,$ one would expect that  the quasi-affinity of $T$ to selfadjoint operators is connected to the Sebestyén inequality involving $T \tstar.$  However, it will be seen in \cref{section04} that   the reversed  inequality    is   more appropriate for such an approach and this will be achieved through a further study of linear relations, which will be discussed in the next section.

 \section{ Generalization to linear relations} \label{secrelation}


In this section an analog of Theorem \ref{firsttheo01} is established for the case where the operator inequality therein is reversed.
For this purpose it is helpful to first prove  Theorem \ref{firsttheo01}  in a bit more general context where $T$ and $B$ are not assumed to be densely defined and, in fact, they will also be allowed to be multivalued linear relations.
This needs some basic facts concerning ordering of semibounded selfadjoint relations; see \cite[Section~5.2]{behrndt2020boundary}
and e.g. \cite{Sandovi2013,hassi22014}.

Before stating the result, some key notions on  linear relations in Hilbert spaces are recalled; for further details, the reader is referred to \cite{cross, behrndt2020boundary, arens}.   A linear relation (relation) $T$ from $\h$  to  $\kk$ is a linear subspace of the Cartesian product $\h \times \kk$. It is uniquely determined by its graph $G(T)=\{ (x,y) \in \h \times \kk : \ x \in \D T , y \in Tx \}.$
Unless otherwise specified, the same notations, familiar for linear operators, will be used for linear relations.  The \emph{inverse} and the \emph{adjoint} of $T$ are respectively given by  $G(T\inv)=\{ (y,x) \ (x,y) \in G(T)\}$ and  $G(T\adj)=\{ (x,x') \in \kk \times \h ; \ \langle x',y \rangle = \langle x, y' \rangle \text{ for all } (y,y') \in G(T)\}$.
For a closed operator $T,$ the  operator part   is given by $T_s=P_{s}T$,
where $P_s$ stands for the orthogonal projection onto $(\mul T)^\perp=\dombar T^*$. Moreover,
  $T_s$ is   closed and $T$ decomposes as $T=T_s \,\widehat\oplus\, T_{\mul}$,
where $T_{\mul}=(\{0\} \times \mul T)$.\\
\hspace*{0.4cm}
If $\kk=\h$ and $\langle x',x \rangle \in \R $ for all $(x,x')\in G(T)$ then $T$ is said to be \emph{symmetric} or, equivalently, $T \inclu \tstar.$
 If  $ \langle x',x \rangle \in \rplus$ then $T$ is \emph{nonnegative} and one writes $T \geq 0.$   Moreover,  $T$ is \emph{selfadjoint} if $T = \tstar.$
 Note that, if $T=\tstar \geq 0$ then $ {\ts}\half := (T_s)\half= (T\half)_s$. For a closed linear relation $T$ the product $\tstar T$ is a nonnegative selfadjoint relation; see \cite[Lemma~1.5.8]{behrndt2020boundary}.
In particular, $T_s\subseteq T$ and $T^* \subseteq (T_s)^*$, so that
\begin{equation}\label{T*Teq}
  T^*T \subseteq (T_s)^*T = T^*P_s T = T^*T_s \subseteq (T_s)^*T_s
\end{equation}
and here all inclusions prevail as equalities, since $T^*T$ and $(T_s)^*T_s$ both are selfadjoint. Notice that if $T$ is a closed operator, which is not densely defined, then $T^*T$ is a selfadjoint relation with $\mul T^*T=(\DT)^\perp.$

\subsection{Sebestyén inequality for linear relations } \label{relsub1}

 The next result allows $T$ and $B$ to be closed linear relations such that the case of densely defined operators in  Theorem \ref{firsttheo01} is explicitly included in it. It should be pointed out that,    exactly as in the case of linear operators; cf. \eqref{motivation}, one has
\begin{equation}\label{mdina}
  \tstar B = \tstar B \rest \D \tstar B = \tstar \overline{B \rest \D \tstar B},
\end{equation}
where $B \rest D := B \cap ( D \times \kk)$ denotes the restriction of the relation $B: \h \rightarrow \kk$ to a linear subspace $D \inclu \h.$

\begin{theorem}\label{relationtheo}
  Let $T, B: \h\rightarrow \kk$ be closed linear relations such that
$\mul B \subseteq \ker (\ts)\adj$ and   $\tstar B $ is selfadjoint.
Then, the following statements are equivalent  for $B_0:= B \rest \D \tstar B$ and for  some $\lambda \geq  0:$
 \begin{enumerate}\def\labelenumi{\rm(\roman{enumi})}
     \item $T^* T  \leq \lambda \, \tstar B ;$
     \item $  X \overline{B_0}\inclu T   $ has a solution $X \in \lplusk.$
          \end{enumerate}
 In this case, $X$ can be chosen such that $X\overline{B_0}  \inclu T_s$ and $ \ker (\ts)\adj \inclu \ker X$ with $\n X \n \leq \lambda.$ Moreover, in this case
      \begin{equation}\label{hawnha}
        \tstar \overline{B_0} =  B_0 \adj X \overline{B_0}= B_0\adj T.
      \end{equation}
 In particular, the following assertions are equivalent for some $\lambda \geq 0:$
   \begin{enumerate}[i)]\def\labelenumi{\rm(\roman{enumi})}\setcounter{enumi}{2}
     \item $T^* T  \leq \lambda \tstar \overline{B_0}$ and $\DT  \inclu \D \overline{B_0}$;
     \item $T=X\overline{B_0}\,\overset{.}{+}\, T_{\mul}$ has a solution $X \in \lplusk.$
      \end{enumerate}
      In this case, $X$ can be chosen such that $T_s=X\overline{B_s}$ and $\ker X =\ker (\ts)\adj.$
\end{theorem}

\begin{proof}
  Observe that  item $(i)$  is equivalent to $(\ts)\adj \ts \leq \lambda \halftsb_s \halftsb_s ,$ and hence the formula
          $$
            \begin{array}{rcl}
                G: \ran   (\tstar B)^{\half}_s & \longrightarrow & \ran T_s \\
                 \halfltsbs  f               & \longmapsto &  T_s f,  \quad f \in \D (\tstar B)^{\frac{1}{2}}_s,
           \end{array}
           $$
   defines a contractive operator from $\ran (\tstar B)^{\half}_s$ into $\ran\ts,$ since
   $$
   \n G  \halfltsbs f \n^2 =
    \n  T_s f \|^2 \leq \n \halfltsbs  f \|^2
    $$
    for all $f \in \D   \halftsbs.$
   Moreover, $G $ can be extended to an operator $G_0 \in \bhk$ such that
   \begin{eqnarray}\label{formulakerg}
     \ker( G_0 )\supseteq  ( \ran     \halftsbs )\orth
      = & \ker     \halftsbs  \oplus \mul\tstar B  \nonumber \\
      = & \hspace*{-0.5cm} \kerr   \tstar B \oplus \mul\tstar B
   \end{eqnarray}
   and    $\ranbar G_0 \inclu \ranbar \ts $, which is equivalent to
   $\ran (G_0)\adj \inclu \ranbar(\tstar B)_s$ and
    \begin{equation}\label{correc}
      \kerr   G_0\adj \supseteq\kerr   (\ts)\adj=  \kerr   \tstar \oplus \mul T.
    \end{equation}
   Thus,
   \begin{equation}\label{otherprof}
     G_0   \halfltsbs \inclu T_s.
   \end{equation}
  As $T$  is closed, $\ts$ is also closed and $\ts \inclu T.$ Hence,
  \begin{equation*}\label{other5}
  \tstar \inclu (\ts)\adj \inclu      \halfltsb   G_0\adj
  \end{equation*}
  which implies that
  \begin{equation}\label{other55}
  \lambda \tstar B \inclu \halfltsb \lambda G_0\adj B.
 \end{equation}
  By assumption $\mul B \inclu \ker (\ts)\adj$, and hence $\mul \tstar B = \mul \tstar T = \mul \tstar.$    On the other hand,   $\mul B \inclu \ker (\ts)\adj \inclu \ker   G_0\adj$ (see \eqref{correc}), so
       \begin{equation} \label{other0}
         G_0\adj B   = G_0\adj  (B_s  \overset{.}{+} B_{\mul})
                          =    G_0\adj  B_s.
        \end{equation}
            This yields by \eqref{other55} that
            %
                 \begin{equation} \label{other}
                  (\lambda  \tstar B)_s \inclu    \halfltsbs G_0\adj B_s.
                   \end{equation}
                   %
        Multiplying \eqref{other} from the left by the Moore-Penrose inverse $  (\tstar B)_s^{(-\frac{1}{2})}$  gives
                 \begin{align*}
                           Q_s (I \rest \D \halftsbs)   \halfltsbs & \inclu   Q_s I \rest \D \halftsbs   \lambda  G_0\adj B_s\\
                                                           & \inclu  \lambda Q_s    G_0\adj B_s = \lambda G_0\adj B_s,
                \end{align*}
               where $Q_s$ is the orthogonal  projection onto $\ranbar (\tstar B)_s.$
               Consequently,
              \begin{equation*}
                \halfltsbs \rest \D \tstar B \inclu \lambda   G_0\adj B_s,
              \end{equation*}
              and hence
              $
                  \halfltsbs| \D \tstar B  = \lambda  G_0\adj B_s \rest \D \tstar B.
              $
             Since $\D \tstar B $ is a core for $ \halftsbs$,   one gets
               \begin{equation*}
               \halfltsbs =   \lambda \overline{G_0\adj B_s \rest \D \tstar B} \supseteq \lambda  G_0\adj \overline{ B_s \rest \D \tstar B}.
              \end{equation*}
              Together with \eqref{otherprof} and \eqref{other0}  this implies that
             \begin{equation}\label{ghayth}
              \lambda G_0 G_0\adj \overline{B_0}= \lambda  G_0 G_0\adj \overline{ B_s \rest \D \tstar B} \inclu  G_0  \halfltsbs \inclu \ts
               \end{equation}
            and, in particular,
              $$
                \lambda G_0 G_0\adj \overline{ B_0} \inclu  T.
                $$
         This proves (ii) for $X=  \lambda  G_0 G_0\adj \in \lplusk.$

                 \hspace*{0.4cm}  The inclusion $\kerr T\adj \inclu \ker (\ts)\adj  \inclu \kerr X$ follows  from \eqref{correc}  by fact that  $\ker (G_0)\adj= \ker X$.

               For the reverse implication $\rm(ii) \Rightarrow (i)$, observe that
   $$
   \tstar \overline{B_0} \inclu B_0\adj X \overline{B_0} \inclu (X\half \overline{B_0})\adj \overline{X\half \overline{B_0}},
   $$
    and since  $ \tstar \overline{B_0} =\tstar B $ is selfadjoint also   $B_0\adj X \overline{B_0}$    is   selfadjoint. Thus,
   \begin{equation}\label{9rib}
      \tstar \overline{B_0}= B_0\adj X \overline{B_0}.
  \end{equation}
  Now, $X\half \overline{X\half \overline{B_0} } \inclu \overline{X\overline{B_0}} \inclu T$ and the same argument that was used in the proof of  Lemma \ref{lemprepthe} shows that for  $\lambda = \n X \n $  one has
  $$
  \tstar T \leq \lambda  (X\half B)\adj \overline{X\half \overline{B_0}} = \lambda  B_0\adj X\overline{B_0}  = \lambda \tstar \overline{B_0},
  $$
  and (i) is proved.

To complete the proof of  \eqref{hawnha}  observe that
 $$
    B_0\adj T \inclu (\tstar \overline{B_0})\adj= \tstar \overline{B_0} = B_0\adj X \overline{B_0} \inclu B_0\adj T,
 $$
  and hence $B_0\adj T= B_0\adj X \overline{B_0} =\tstar \overline{B_0}.$\\
   \hspace*{0.4cm}  For the proof of the equivalence $\rm(iii) \ifaf (iv)$, it suffices to observe that item (iv) is   equivalent to  $
  X \overline{B_0} \inclu \ts  $ and $ \DT = \D \overline{B_0},$ and conclude the result from the first part of the proof. In this particular case, one easily sees that $\ts =X \overline{B_0} $ and hence $(\ts)\adj = B_0\adj X$, which leads to  $\ker X \inclu \ker  (\ts)\adj \inclu \ker X.$
\end{proof}

As seen in Remark \ref{remark1}, one obtains from \eqref{hawnha} the following inequality
$$
 \tstar B \leq  \mu  (B_0)\adj   \overline{B_0}, \qquad \mu = \n X \n,
$$
which implies that $\D \overline{B_0} \inclu \D \halftsb \inclu \DT.$  Some further properties of  $T$ and $B$ are collected in  the next remark.
 \begin{remark}
 \normalfont Under the assumptions of Theorem \ref{relationtheo} the following further statements hold:
\begin{enumerate}
 \item $ (\ts)\adj B =(T_s)^*B_s.$
 \item    $ \mul T^*B = \mul \tstar  $ (equivalently, $ \dombar \tstar B=\dombar T$);
 \item if $ \overline{B_0}= B$ then $\mul T \cap \D B\adj \inclu \ker B\adj \inclu \ker (B_s)\adj;$
\item  As noted above $T^*T=(T_s)^*T_s$; cf. \eqref{T*Teq}. Likewise, if $ \overline{B_0}= B$ then
      \begin{equation}\label{abersab}
        (\ts)\adj B_s  =  (B_s)\adj \ts,
      \end{equation}
   which implies that
      \begin{equation}\label{abersab2}
        \tstar B = (\ts)\adj B= (\ts)\adj B_s=  (B_s)\adj X  B_s,
      \end{equation}
      where $X \in \lplus.$
\item If $ \overline{B_0}= B$ then the first item of Theorem \ref{relationtheo} can be written with the operator part of $T$:
      $$
      (\ts)\adj \ts \leq \lambda \, (\ts)\adj \overline{B_0}.
      $$
      \end{enumerate}
 \hspace*{0.4cm} Indeed, the identity  \rm (1) follows easily  from the inclusion $\mul B \inclu \ker (\ts)\adj=  \ker  \tstar \oplus \mul T$ which implies that
        \begin{equation*}
        (\ts)\adj B =T^* P_s \left(B_s \,\widehat\oplus\, B_{\mul}\right)= T^*P_s B_s =(T_s)^*B_s.
       \end{equation*}
       \hspace*{0.4cm}  To see  \rm (2),    apply   \rm{(1)} to get
\begin{equation*}
 \mul T^* \inclu \mul T^*B\subseteq \mul (T_s)^*B = \mul (T_s)^*B_s=\mul T^*.
\end{equation*}
Hence   $\mul T^*=\mul  T^*B$ or, equivalently, $ \dombar T  = \dombar \tstar B.$\\
       \hspace*{0.4cm} For the proof of \rm (3), observe that $\mul B\adj \inclu \mul B\adj T \inclu \mul (\tstar B)\adj =  \mul \tstar B \adj.$ On the other hand,    $\mul \tstar B  \inclu \mul B\adj,$ by   Remark \ref{remark1} (i).   Hence, $ \mul B \adj = \mul B\adj T,$ which means that
       $$ \mul T \cap \D B\adj \inclu \ker B\adj \inclu \ker (B_s)\adj.
       $$
       \hspace*{0.4cm} For the proof of \rm (4), observe that   $X B = X B_s \inclu \ts$ together with \eqref{hawnha} and item (1)  yields
       \begin{equation}\label{bye1}
         (\ts)\adj B_s = (\ts)\adj B  \inclu (X B)\adj B = B\adj X B = \tstar B  \inclu (T_s)\adj B_s.
       \end{equation}
       This means that $(\ts)\adj B_s$ is selfadjoint and, moreover,
       \begin{equation}\label{bye2}
       (\ts)\adj B_s \inclu  (B_s)\adj X B_s \inclu  (B_s)\adj \ts \inclu ( (\ts)\adj B_s)\adj = (\ts)\adj B_s.
         \end{equation}
        A combination of \eqref{bye1} and \eqref{bye2} shows \eqref{abersab} and \eqref{abersab2}.\\
       \hspace*{0.4cm}  To see (5), observe from \eqref{abersab2} and \eqref{mdina} that
       $$
       (\ts)\adj B= \tstar B = \tstar \overline{B_0} \inclu (\ts)\adj \overline{B_0}  \inclu  (\ts)\adj B,
       $$
       which implies that  $\tstar B =  (\ts)\adj \overline{B_0}.$ Together with \eqref{T*Teq}, this implies that
       $$
      T^* T  \leq \lambda \, \tstar B  \ \ifaf  (\ts)\adj \ts \leq \lambda \, (\ts)\adj \overline{B_0}.
      $$
\end{remark}

\subsection{Characterization of the reversed inequality} \label{subrelinv}

The following result  shows that reversing Sebestyén inequality yields  a new nonnegative, in general,  unbounded solution $X$ with $X\inv \in \lplus$ rather than a bounded one as seen in Theorem \ref{firsttheo01} and Theorem \ref{relationtheo}. This  motivates the study of a new unbounded product   of nonnegative operators; see \cref{section04}.

\begin{theorem}\label{corsebestyeninclusion}
Let $\kk$ be a complex Hilbert space and $T,B:\h \rightarrow \kk $ be closed linear relations such that $B\adj T$ is selfadjoint and  let  $B_0 := B\adj \cap ( \kk \times \ran B\adj T ).$ If  $\ker B\adj \subseteq \ker T^* \oplus \mul T$ then the following assertions are equivalent for some $\eta > 0:$
   \begin{enumerate}[(i)]\def\labelenumi{\rm(\roman{enumi})}
   \item  $T^* T  \geq \eta  \,  \overline{B_0} T \geq 0;$
     \item $T \inclu Y B_0\adj $ has a solution $Y\inv \in \lplusk.$
      \end{enumerate}
      In this case, $Y$ can be chosen such that $\ker \tstar \oplus \mul T  \inclu \mul Y$ and
      \begin{equation}\label{agosolo}
    B\adj T=  \overline{B_0} T  =  \overline{B_0} Y {B_0}^{*} =  \tstar {B_0}^{*}
      \end{equation}
      In particular, the following statements are equivalent for some $\eta >0:$
   \begin{enumerate}[i)]\def\labelenumi{\rm(\roman{enumi})}\setcounter{enumi}{2}
     \item $T^* T  \geq \eta \,  \overline{B_0} T$ with $\ran{\tstar} \inclu    \ran \overline{B_0};$
     \item $\tstar = \overline{B_0}  Y  \overset{.}{+}  ( \ker \tstar \times \{0\})$ has a solution $Y\inv \in \lplusk.$
      \end{enumerate}
      In this case, $\mul T \oplus \ker \tstar = \mul Y.$
      \end{theorem}

\begin{proof}
First observe that
\begin{eqnarray}   \label{jamlaaadi}
B\adj T =  B_0 T = \overline{B_0} T
\end{eqnarray}
is   selfadjoint.
Now, let  $S:={\tstar}\inv$  and $A:=(B\adj)\inv$. Then, $S$ and $A$ are two closed linear relations such that
 $
(S\adj A)\adj = ( T\inv ({B}\adj)\inv)\adj = ({({B}\adj T)\adj})\inv=({B}\adj T)\inv= S\adj A$
and the assumption $\ker {B}\adj \inclu  \ker T^* \oplus \mul T$ is equivalent to $\mul{A} \inclu \mul{S} \oplus \ker{S\adj}  = \ker (S_s)\adj.$
Now,  using Remark \ref{remark1} (iii), one can apply  Theorem  \ref{relationtheo} to $S$ and $A$ which yields  the following equivalences for   $\lambda > 0$ and $A_0:= A \rest \D S\adj A$:  %
\begin{enumerate}
   \item $S^* S  \leq \lambda \, S\adj \overline{A_0};$
     \item $  X \overline{A_0} \inclu S   $ has a solution $X \in \lplusk,$  $X \neq 0,$
\end{enumerate}
where  $X$ can be chosen such that $\ker (S_s)\adj  \inclu \ker X$  and
     \begin{equation}\label{jamlaaadi2}
         S\adj \overline{A_0} =  A_0 \adj X \overline{A_0}= A_0\adj S  .
     \end{equation}
Equivalently,  $ \mul T   \oplus \ker \tstar \inclu \mul Y $   for $Y = X\inv .$  By taking inverses  the equalities \eqref{jamlaaadi2} can be rewritten as
     $$
    \overline{B_0} T  =  \overline{B_0} Y {B_0}^{*} =  \tstar {B_0}^{*}
     $$
using the fact that $ B_0=   \left( A_0 \right)\inv .$
Combining this with \eqref{jamlaaadi} proves \eqref{agosolo}.
Next, using \cite[Lemma $3.3$]{hassi2006form}, or \cite[Corollary 5.2.8]{behrndt2020boundary} one has the following equivalence for some $\eta= \frac{1}{\lambda} > 0:$
\begin{enumerate}
   \item $(S^* S)\inv  \geq  \eta \, (S\adj \overline{A_0})\inv ;$
   \item $  (X \overline{A_0})\inv =\overline{A_0}\inv X\inv  \inclu S\inv   $ has a solution $X \in \lplusk,$ $X \neq 0.$
\end{enumerate}
By taking adjoints in (2) this equivalence can be rewritten as 
\begin{enumerate}
   \item $ \tstar T   \geq  \eta \,  \overline{B_0} T ;$
   \item $  T \inclu Y B_0\adj  $ has a  solution $Y\inv \in \lplusk.$
\end{enumerate}
Next, to see the equivalence $\rm (iii) \ifaf (iv)$, observe that \rm (iii) is equivalent to $S^* S  \leq \lambda \, S\adj \overline{A_0}$ and $\D S \inclu \D \overline{A_0},$ which is equivalent to $S= X \overline{A_0} \overset{.}{+} S_{\mul},$ by  Theorem \ref{relationtheo}. This last identity can be rewritten in the form
\begin{equation*}\label{chbilia}
  (\tstar)\inv = Y\inv {\overline{B_0}}^{\, -1}  \overset{.}{+} ( \{ 0 \} \times \ker \tstar )
 \ifaf  \tstar = \overline{B_0} Y \overset{.}{+} ( \ker \tstar  \times  \{ 0 \}).
\end{equation*}
Furthermore, it follows from Theorem   \ref{relationtheo} that $\ker S\adj \oplus \mul S = \ker X,$ which means that $\mul T \oplus \ker \tstar = \mul Y.$
\end{proof}

The following result is   analogous to the first items of Remark \ref{remark1} and Remark \ref{remarko5ra}.

\begin{remark}\normalfont
Under the assumptions of Theorem \ref{corsebestyeninclusion}, one obtains from  \eqref{agosolo} the following  implication
%
\begin{eqnarray}
&T \inclu Y B_0\adj  \text{ has a solution } Y\inv \in \lplusk
    \nonumber \\
& \Downarrow\label{iko}   \\
& \hspace*{2cm} B\adj T \geq \mu  \overline{B_0} {B_0}\adj, \quad \mu = \tfrac{1}{ \n Y\inv \n}. \nonumber
\end{eqnarray}
In particular, if (the graph of) $B_0$ is a core of $B\adj,$ i.e.  $\overline{B_0}=B\adj,$  then the converse implication in \eqref{iko} holds, i.e.,
\begin{eqnarray}
&T \inclu Y B  \text{ has a solution } Y\inv \in \lplusk
    \nonumber \\
&  \Updownarrow \label{iko2}   \\
& \hspace*{1.6cm} B\adj T \geq \mu    {B}\adj B, \quad \mu = \tfrac{1}{ \n Y\inv \n}. \nonumber
\end{eqnarray}
\end{remark}

\begin{remark}\normalfont
The equivalence stated in  \eqref{iko2} can also be established    under conditions different  from those given in Theorem  \ref{corsebestyeninclusion}, in particular, when  $B\adj T \geq 0$ is selfadjoint,  $\mul T \inclu \ker (B_s)\adj$ and  $\D B\adj T$ is a core for the operator part $B_s. $
To see this, it suffices reverse the roles of $B$ and $T$ in Remark \ref{remarko5ra} and observe that $ T \inclu YB \ifaf Y\inv T \inclu B$ for any $Y \inv \in \lplusk.$
\end{remark}

The following corollary treats  a  particular case of Theorem  \ref{corsebestyeninclusion}, the case of densely defined operators with dense ranges.

\begin{corollary}\label{inversesebestyen}
Let $\kk$ be a complex Hilbert space and $T,B:\h \rightarrow \kk$ be closed densely defined linear operators    such that    $B\adj T$ is selfadjoint and $\overline{B_0}=B\adj .$
If $\ranbar T = \ranbar B=\h $  then the following assertions are equivalent for some $\eta > 0:$
\begin{enumerate}[(i)]\def\labelenumi{\rm(\roman{enumi})}
   \item  $T^* T  \geq \eta \, B\adj T \geq 0 $ with  $\ran{\tstar} \inclu    \ran B\adj;$
   \item $ \tstar =   B\adj \, Y $ has a solution $Y\inv \in \lplusk.$
\end{enumerate}
\end{corollary}

\section{The class $\lldeux$ and the  reversed   inequality} \label{section04}

  In this section, the emphasis will be on the following class
  \begin{equation} \label{class2}
    \lldeux:=\{ T=AB ; \ A\inv \in \lplus,\ B= B\adj \geq 0 \}
  \end{equation}
as a modification of the class $\ldeux.$ In \eqref{class2},  $A$ is invertible, i.e., belongs to the class   $Gl(\h)$ of closed densely defined  injective and onto  operators on $\h.$ Denote by     $GL(\h)$  the set of all bounded everywhere defined invertible operators and, moreover, one      has $\gplus:= GL(\h) \cap \lplus$ and   $Gl^+(\h):=\{ S \in Gl(\h); \ S= S\adj \geq 0\}.$ Note that $ S \in Gl^+(\h)$ if and only if $S$ is a nonnegative selfadjoint operator with $\ran S= \h.$

The simpler case where $T$ belongs to $\lldeux  \cap  \ldeux$  will be treated in \cref{wsimilsec} and involves  weak similarity as well as   similarity  to nonnegative selfadjoint operators, while the general case, treated in \cref{subsectionlldeux}, is rather connected to   quasi-affinity and quasi-similarity to nonnegative selfadjoint operators. These notions will appear to be significantly related to the reversed inequality treated in \cref{subrelinv}.\\

\subsection{Similarity and $\mathcal{W}$-similarity to $S=S\adj \geq 0$ } \label{wsimilsec}

An operator $T \in LO(\h)$ is said to be $\mathcal{W}$-\emph{similar}  to $S \in LO(\h)$ if there exists $G \in GL(\h)$   such that $$
GT \inclu SG.
$$
If $TG=GS$  then $T$ is \emph{similar} to $S.$  In particular, if $T$ is similar to a normal  operator then it is said to be \emph{scalar}; see    \cite{bade1954unbounded, dunfordthree,scalar} for general background on scalar operators.  The next proposition  characterizes  $\mathcal{W}$-similarity to nonnegative selfadjoint operators with non-empty resolvent set.

\begin{proposition}\label{theosimilarscalar}
   Let $T \in CO(\h)$ be a densely defined linear operator. Then, the following assertions are equivalent:
   \begin{enumerate}\def\labelenumi{\rm(\roman{enumi})}
       \item $T$ is $\mathcal{W}$-similar to a  nonnegative selfadjoint  operator $S$ with $\rho(T) \neq \vide;$
       \item $XT=  \tstar X$, where $X \in \gplus$ and $\sigma(T) \inclu \rplus$;
       \item   $T= X_1 B_1  \ $  with $X_1 \in \gplus$ and $B_1=B_1\adj \geq 0$    (respectively,  $\tstar=X_2B_2$ with $X_2 \in \gplus$ and $B_2=B_2\adj \geq 0);$
       \item $T=BX$,  where $B=B \adj \geq 0 $ and $X \in \gplus$
           (respectively, $\tstar= B' Y$ with $B'=(B')\adj \geq 0$ and $Y \in \gplus$);
       \item There exist $W, Z \in \gplus$  such that   $T W =(TW)\adj \geq 0$ (respectively, $Z T=(ZT)\adj \geq 0);$
       \item $T$ is a scalar operator and $\sigma(T) \inclu \rplus.$
   \end{enumerate}
    If one of the above conditions holds, then
       \begin{equation}\label{plusdot}
          \ranbarT  \overset{.}{+} \kerT=\h.
       \end{equation}
\end{proposition}
\begin{proof}
      $\rm (i) \Rightarrow (ii)$
       Since $T$ is $\mathcal{W}$-similar to a nonnegative operator $S,$ there exists $G \in GL(\h)$ such that $GT\inclu SG.$ Hence,
       $$
       GT G\inv \inclu S = S\adj   \inclu (GT G\inv)\adj,
       $$
     which shows that $GT G\inv$ is symmetric. As $G\in GL(\h),$ one then has  $\rho(GT G\inv)=\rho(T) \neq \vide,$  and therefore
      \begin{equation}\label{smeh1}
       GT G\inv = S = (GT G\inv)\adj= {G\inv}\adj \tstar G\adj.
       \end{equation}
   This yields that
      $$
     G\adj G T  =  \tstar G\adj G,
      $$
   and the statement follows by taking $X=G\adj G \in \lplus.$ Furthermore, $(\ref{smeh1})$ shows that $\sigma(T)= \sigma(GT G\inv)= \sigma(S) \inclu \rplus.$ \\
        \hspace*{0.3cm}
       $\rm (ii) \Rightarrow (iii)$
        Let $T = X\inv  \tstar X, $ $ X \in \gplus,$ and assume that $\sigma(T) \inclu \rplus.$  Then,
        $
        X\half T X\mhalf =  X\mhalf \tstar X\half,
        $
        and hence
         $$
        ( X\half T X\mhalf)\adj=  X\mhalf \tstar X\half= X\half T X\mhalf.
        $$
Since $X  \in GL(\h)$ and $ \sigma(T) \inclu \rplus,$ it follows that
$
\sigma(X\half T X\mhalf)=\sigma(T) \inclu \rplus,
$
and therefore
\begin{equation}\label{eq2boundedinv}
S:=X\half T X\mhalf = X\mhalf \tstar X\half  =S\adj \geq 0.
\end{equation}
 Thus,
 $$
   B_1:=  X\half S  X\half= X\half  (X\half T X\mhalf ) X\half= X T = {B_1}\adj \geq 0
    $$
    and  $T =  X\inv  B_1 =X_1 B_1$, where $X_1=X\inv $ is invertible.   \\
 \hspace*{0.4cm}  To prove the remaining statement, observe from $(\ref{eq2boundedinv})$ that
   $$
B_2:=  X\mhalf S  X\mhalf= X\mhalf ( X\mhalf \tstar X\half) X\mhalf =X\inv \tstar= B_2\adj  \geq 0$$
and    $\tstar=XB_2$ with $X$ invertible. \\
     \hspace*{0.4cm} The equivalence $\rm (iii) \Leftrightarrow (iv)$ is direct.\\
          \hspace*{0.4cm}     $\rm (iii) \Rightarrow (v)$     Assume that
          $T=X_1B_1 $ with $X_1 \in \gplus.$ Then, for $Z:= {X_1}\inv \in \gplus $ one has     $ZT= B_1= {B_1}\adj= (ZT)\adj \geq 0.$
         \\        %
        \hspace*{0.4cm}  Similarly,    $\tstar=X_2 B_2\in \ldeux$ with $ X_2\in \gplus$ and $W:= X_2\inv \in \gplus$ yield that  $T=B_2X_2$ and  $T W =B_2= B_2\adj=(TW)\adj \geq 0.  $\\
        \hspace*{0.4cm} $\rm (v) \Rightarrow (vi)$
        Assume that there exists $W \in \gplus$  such that  $S_0=T W={S_0}\adj  \geq 0.$
           Then, $ W\mhalf S_0 W\mhalf \geq 0,$ $W\half \in GL(\h)$, and one has
                 $$
                  W\half ( W\mhalf S_0 W\mhalf )  =  T W\half .
                 $$
                Similarly if   $Z \in \gplus$ such that $S_1=Z T = {S_1}\adj \geq 0,$ then $Z\mhalf \in GL(\h),$ $Z\mhalf S_1  Z\mhalf = (Z\mhalf S_1  Z\mhalf )\adj \geq 0$  and
                 $$
                 T Z\mhalf = Z\mhalf (Z\mhalf S_1  Z\mhalf).
                 $$
                In both cases, one concludes that $T$ is similar to a nonnegative selfadjoint  operator   and    $\sigma(T) =\sigma(S_0)=\sigma(S_1) \inclu \rplus $. By definition $T$ is a scalar operator.\\
          \hspace*{0.4cm} $\rm (vi) \Rightarrow (i)$ If $T$ is a scalar operator with  $\sigma(T) \inclu \rplus$ then it is easily seen that it is similar, and hence $\mathcal{W}$-similar to a nonnegative selfadjoint operator. \\
      \hspace*{0.4cm} If one of the above conditions holds, then $T$ is   similar to $S=S\adj$ and $(\ref{plusdot})$ follows directly from  $\ranbar S \overset{.}{+} \ker S=\h$.
       \end{proof}

\begin{remark} \normalfont  Note that in Proposition \ref{theosimilarscalar}, the similarity and the $\mathcal{W}$-similarity to a nonnegative  selfadjoint operator are the same, cf. \eqref{smeh1}.
\end{remark}

\subsection{$\lldeux$ and quasi-affinity to $S=S\adj \geq 0$}\label{subsectionlldeux}

Recall that in \Cref{subsection1}, the quasi-affinity of $\tstar$ to a nonnegative selfadjoint operator $S$   is described  through elements $T$ in $\ldeux$. Unlike in the case of bounded operators, the quasi-affinity of $\tstar$ to $S$ does not  imply the one of $T$. The latter will rather  be described by elements of $\lldeux$ in the  following theorem.

\begin{theorem}\label{TBAsurjective}
   Let $T\in CO(\h) $ be densely defined. Then the following statements are equivalent:
   \begin{enumerate}\def\labelenumi{\rm(\roman{enumi})}
       \item   $T$ is  quasi-affine to some $S=S\adj \geq 0;$
       \item $T \inclu B A  \in \lldeux ;$
        \item there exists a quasi-affinity $X \in \lplus$  such that
         $$
      0 \leq X T  \inclu  \tstar X.
        $$
   \end{enumerate}
\end{theorem}

\begin{proof}
    $   \rm (i)\Rightarrow  (ii)$ Assume   that $T$ is $G$-quasi-affine to $S= S\adj \geq 0$ and  fix         $A_0:=  G\adj SG $ and $B:= ( G\adj G)\inv .$  Then $B\inv \in \lplus $ and   the inclusion $G T \inclu S G$ implies that $ B\inv T =G\adj G T   \inclu G \adj  S G =A_0 \geq 0 $ with $ \dombar A_0 =\h.$   Let now   $A_F = {A_F}\adj \geq 0$ be the Friedrichs extension of $A_0.$   Then
      \begin{equation}\label{dima}
        B\inv T \inclu  A_0 \inclu A_F,
      \end{equation}
     and, therefore, $ T  \inclu B A_F \in \lldeux.$ \\
        \hspace*{0.4cm} $\rm (ii) \Rightarrow (iii)$  Since $T \inclu BA \in \lldeux$ it follows that $ B\inv T \inclu A \inclu (B\inv T)\adj= \tstar B\inv.$ Hence, for $X=  B\inv \in \lplus$ one has
        $   0 \leq X T  \inclu  \tstar X.$\\
      \hspace*{0.4cm} $\rm (iii) \Rightarrow (i)$ Let $A_0 := XT \geq 0.$ Then $A_0$ is densely defined. Let $A= A\adj \geq 0$  be a selfadjoint extension of $A_0$. Then clearly
     \begin{equation}\label{1003}
        XT \inclu A \inclu (XT)\adj = \tstar X.
       \end{equation}
       Now let $S_0:= X\mhalf A X\mhalf \geq 0$. Then by multiplying  \eqref{1003}  from the left  and right by $X\mhalf$ and   one obtains
      \begin{equation}\label{hospi}
     0 \leq   X\half T X\mhalf    \inclu  S_0  
        \end{equation}
      Since $\dombar T X\mhalf = \ranbar X = \h$ one concludes that $S_0$ is densely defined operator with the  Friedrichs extension $S_F=  (X\mhalf {A}\half) \overline{{A}\half X\mhalf} .$  Multiplying \eqref{hospi} from the right by $X\half$ one gets
        \begin{equation}\label{etoile}
             X\half T      \inclu S_0 X\half \inclu S_F X\half,
        \end{equation}
       which proves the quasi-affinity of $T$ to $S_F.$
        \end{proof}

\begin{remark} \label{e5rrem} \normalfont
In  the proof of Theorem \ref{TBAsurjective}, \rm(i),   $ A_0 = G\adj S G  = (G\adj S\half) S\half G$ with    $\dombar A_0= \h.$ Hence     $\dombar S\half G =\h$ and one has
\begin{equation}\label{nchlh}
 0 \leq A_0 = G\adj S G  \inclu (S\half G)\adj S\half G = \overline{ G\adj S\half } S\half G= A_F,
\end{equation}
where $A_F$ is the Friedrichs extension of $A_0.$  The proof also works for any    nonnegative selfadjoint extension of $A_0,$  respectively,  $S_0$ (see \eqref{etoile}).
\end{remark}

The following result is the analog of Corollary \ref{remsebsqua}. It shows a connection between the  reversed  inequality and quasi-affinity to a nonnegative selfadjoint operator.

\begin{corollary}\label{revsebquasiaff}
Let $T \in CO(\h)$ be a densely defined operator   let $S=S\adj \geq 0 $. If $T$ is $G$-quasi-affine to $S$  such that $\rho(A_F T) \neq \vide$,  then
 $$ T^* T  \geq  \tfrac{1}{\lambda}  A_F T \geq 0 $$
for some $\lambda >0,$ where $A_F$ is given in \eqref{nchlh}.
\end{corollary}

\begin{proof}
Since $T$ is $G$-quasi-affine to $S$ one obtains from  Theorem  \ref{TBAsurjective} that $T \inclu BA_F$ where  $A_F = A_F\adj \geq 0$ and $B \inv =  G\adj G \in \lplus. $   Hence $B\inv T \inclu A_F$ and $A_F \inclu  (B\inv T)\adj =\tstar B\inv.$ Consequently,
$$
A_FT \inclu  \tstar B\inv  T =\tstar B\mhalf B\mhalf T \inclu \tstar B\mhalf (\tstar B\mhalf)\adj= :F \geq 0.
$$
Since $F= F\adj,$ it follows that  $A_FT$ is symmetric. On the other hand, $\rho(A_FT)  \neq \vide$ by assumption, and therefore
\[
 A_FT = \tstar B\inv T \leq \n B\inv \n  \tstar T \ \Rightarrow \tstar T \geq \tfrac{1}{\lambda}   A_FT \quad 
\text{with } \lambda =  \n B\inv \n. \qedhere
\]
\end{proof}

Note that a small adjustment to item (i) of Theorem \ref{TBAsurjective} allows
$T$ to be written as the product of two nonnegative, in general, unbounded  linear operators  motivating the following result.

\begin{proposition}\label{cond1equiv}
  Let $T\in CO(\h) $ be densely defined. Then, the following are equivalent:
   \begin{enumerate}\def\labelenumi{\rm(\roman{enumi})}
       \item   $T$ is  $G$-quasi affine to some $S=S\adj \geq 0$ such that $ \D \overline{G\adj S\half} S\half G \inclu \DT;$
       \item $T = B A  \in \lldeux  $ and $ \D T = \DA;$
        \item there exists  a quasi-affinity  $X \in \lplus$ such that
         \begin{equation}\label{103}
          X T = \tstar X \geq 0.
        \end{equation}
   \end{enumerate}
\end{proposition}

 \begin{proof}
   $\rm (i) \Rightarrow (ii)$
 Using the same argument as in the proof of Theorem  \ref{TBAsurjective} combined with  Remark \ref{e5rrem} one obtains  that  $B\inv T \inclu A_F= \overline{ G\adj S\half } S\half G,$ cf. \eqref{nchlh}. Now the assumption      $\D A_F \inclu \DT $ shows that
 $$
   B\inv T  = A_F.
   $$
   Hence $T = B A _F \in \lldeux$ and $\D T = \D A_F.$\\
 $\rm (ii) \Rightarrow (iii)$ Observe that $ B\inv T \inclu A $ and since $\DT = \DA$ one obtains
 $$ A = B\inv T = A\adj = \tstar B\inv.
 $$
 Now take $X= B\inv \in \lplus $ to get  \eqref{103}. \\
 $\rm (iii) \Rightarrow (i)$ By assumption $A=XT \geq 0$ is selfadjoint. 
 Now proceed as in the proof of Theorem \ref{TBAsurjective}. Then the operator
 \begin{equation}\label{hospoto}
   S_0= X\mhalf  A X\mhalf = X\half T X\mhalf \geq 0
 \end{equation}
is densely defined where its Friedrichs extension $S_F$ satisfies \eqref{etoile} and $T$ is quasi-affine to $S_F.$ Multiplying \eqref{etoile} from the left by $X\half$ gives
   $$
   X T  = X\half S_0  X\half \inclu X\half S_F  X\half \inclu  E_F,
   $$
   where $ E_F= \overline{X\half {S_F}\half} {S_F}\half X\half $ denotes the Friedrichs extension of $X\half S_F X\half.$ Consequently, $XT = E_F$ and  $\D E_F = \DT,$ as required.
 \end{proof}

It is worth noticing that the quasi-affinity of $T$ together with that of $\tstar$ gives raise to a new notion defined below, which will be characterized in Lemma \ref{equivquasisimilarity}.

 \begin{definition} \normalfont \cite[Definition 2.1]{Ota89}
  $T \in LO(\h)$ is said to be
  \emph{quasi-similar}    to $S \in LO(\h)$ if there
   exist two quasi-affinities $G_1 , G_2 \in B(\h)$ such that
$$
G_1T \inclu SG_1 \quad \text{and} \quad G_2S \inclu TG_2.
$$
\end{definition}

The next lemma contains a duality property of the quasi-affinity and characterizes the quasi-similarity to a nonnegative selfadjoint operator.

\begin{lemma}\label{equivquasisimilarity}
Let $T,S  \in CO(\h)$ be a densely defined operators.   Then the following statements are
equivalent:
\begin{enumerate}[\rm(i)]
 \item  $T$ is $G$-quasi-affine to $S$ $\ifaf$ $S\adj$ is $G\adj$-quasi-affine to $\tstar;$
  \item  $T$ is quasi-similar to   $S=S\adj $ $\ifaf$ $T$ and $\tstar$ are quasi-affine       to $S=S\adj.$
\end{enumerate}
\end{lemma}

\begin{proof}
\begin{enumerate}[\rm(i)]
\item Let $S \in CO(\h).$ Then $T$ is $G$-quasi-affine to $S$ $\ifaf$ $G T \inclu SG $ $\ifaf $ $ G\adj S\adj \inclu \tstar G\adj$, i.e. $S\adj$ is $G\adj$-quasi-affine to $\tstar.$
\item   If $T$ is   quasi-similar to $S,$ then there are two quasi-affinities $G_1 , G_2 \in B(\h)$
      such that $ G_1T \inclu SG_1$ and $G_2S  \inclu   TG_2.$ This shows that $T$ is $G_1$-quasi affine to $S$ and, by (i), $\tstar$ is $G_2\adj$-quasi-affine to $S.$
Conversely, if   $T$ and $\tstar$ are quasi-affine to $S,$  then it follows from $\rm{(i)}$ that there are   two
      quasi-affinities $G_1,G_2 \in B(\h)$ with the property that $G_1T \inclu SG_1$
      and $G_2\adj S    \inclu S G_2\adj.$   As $G_2\adj$ is a quasi-affinity, one concludes that $T$ is
      quasi-similar to $S.$
\end{enumerate}
\end{proof}

The next result is now a consequence of Lemma \ref{equivquasisimilarity}, Theorem \ref{TBAsurjective} and Proposition \ref{cond1equiv}.

 \begin{corollary}
Let $T \in CO(\h)$ be a densely defined operator and let $S =S\adj \geq 0.$ If
$ T$ is quasi-similar to $S$
     then  there exist $T_1 \in \ldeux$ and $T_2 \in \lldeux$ such that
   $$
   T_1 \inclu T \inclu T_2.
   $$
   In particular, if
$ T$ and $\tstar$ are respectively     $G_1$ and $G_2$-quasi-affine to $S$  such that
  $\D ( \overline{G_1\adj S\half}  S\half G_1) \inclu \DT \inclu \D (    {{G_2}\adj}\inv S\half \overline{S\half {{G_2}\adj}\inv})$, then
  \begin{equation}\label{sevnap0}
    T \in \ldeux \cap \lldeux.
  \end{equation}
 \end{corollary}

\begin{proof} Assume that $T$ is quasi-similar to $S=S\adj \geq 0.$ Then, by   Lemma \ref{equivquasisimilarity}, there exist two quasi-affinities  $G_1, G_2 \in B(\h)$ such that
  $T$   and $\tstar$  are respectively $G_1$ and $G_2$ quasi-affine to  $S.$ A direct application of  Proposition  \ref{inclusionnfs} and Theorem \ref{TBAsurjective} implies the existence of $A_1\in \lplus$, $B_1={B_1}\adj \geq 0,$  $B_2\inv \in \lplus$ and  $A_2 ={A_2}\adj \geq 0$ such that
 \begin{equation}\label{sevnap}
   \ldeux \ni  A_1 B_1 \inclu T \inclu B_2 A_2   \in \lldeux.
 \end{equation}
Now, assume   that  $\D ( \overline{G_1\adj S\half}  S\half G_1) \inclu \DT \inclu \D (    {{G_2}\adj}\inv S\half  \overline{S\half {{G_2}\adj}\inv})$.
Then equalities hold in \eqref{sevnap} by Propositions \ref{egalquasiaff} and \ref{cond1equiv}, which proves \eqref{sevnap0}.
\end{proof}

\subsection{Quasi-affinity and quasi-similarity to   $\mathbf{S \in \lplus}$} \label{subseection43}

It is worth noticing  that if $T$ is $\mathcal{W}$-similar or similar to a bounded nonnegative  operator $S \in \lplus$, then also $T$ itself is bounded. In this case, its similarity to $S$ is already dealt with in  \cite[Theorem 3.1]{product2021}. The focus is therefore on quasi-affinity and quasi-similarity.

\begin{proposition}\label{ivaa}
  Let $T \in LO(\h)$ be a densely defined operator. Then  the following statements are equivalent:
\begin{enumerate}[(i)]\def\labelenumi{\rm(\roman{enumi})}
\item \textcolor[rgb]{0.00,0.00,0.00}{$T$ is $G$-quasi-affine to $S \in \lplus;$}
  \item  $\overline{GT G\inv}= {G\inv}\adj \tstar G\adj \in \lplus $ for a quasi-affinity $G \in B(\h);$
  \item $ \overline{ X\half T X\mhalf} = X\mhalf \tstar X\half \in \lplus$  for a quasi-affinity    $X  \in \lplus$.
\end{enumerate}
In this case, $T \inclu  B A$,    where $A, B\inv \in \lplus.$   Moreover, there exists a   quasi-affinity $X \in \lplus$  such that
        \begin{equation}\label{halm}
          \tstar X =  \overline{XT} \in \lplus .
        \end{equation}
\end{proposition}

\begin{proof}
$\rm{(i)}  \Rightarrow \rm{(i)}$ Observe that the inclusion $GT \inclu SG$ implies that $GTG\inv \inclu S.$ Since $S \in B(\h)$ and $\dombar GTG\inv  =\h$ one concludes that
\begin{equation}\label{44}
   S= (GTG\inv)\adj = ({G\inv})\adj \tstar G\adj = \overline{G T G\inv} \in \lplus.
\end{equation}
  $\rm{(ii)} \Rightarrow  \rm{(i)}$ Fix $S_0:=\overline{GT G\inv}.$ Then, $S_0 \in \lplus$ and $GT = G T G\inv G \inclu  \overline{G T G\inv} G  \inclu S_0G,$ as required. \\
      \hspace*{0.3cm}  The equivalence $\rm{(ii)} \ifaf  \rm{(iii)}$ follows directly from Lemma \ref{remquasiaffini}.

Now, assume that $T$ is $G$-quasi-affine to $S \in \lplus.$ Then $G\adj G T \inclu  G\adj SG =:A \in \lplus,$ and hence $ T \inclu B A$ with $B\inv:= G\adj G \in \lplus.$ To see \eqref{halm}, observe from \rm(iii) that for $M:= X\mhalf \tstar X\half \in \lplus$ one has
$X\half M X\half \inclu \tstar X,$ which yields that
\[
 \overline{XT }= (\tstar X)\adj = (X\half M X\half)\adj
                                       = X\half M X\half
                                       \in \lplus. \hfill \qedhere
\]
\end{proof}

The following theorem is the optimal analogue of Corollary \ref{revsebquasiaff} in the context of the reversed inequality.

\begin{theorem}\label{lasttheo}
 Let $T\in CO(\h) $ be a densely defined operator. If $T$ is $G$-quasi-affine to some $S  \in \lplus$ then exists $A \in \lplus $  such that
        \begin{equation}\label{inequalitychoice}
           T^* T  \geq  \tfrac{1}{\lambda}    \overline{AT}
        \end{equation}
      for some $\lambda > 0.$
\end{theorem}

\begin{proof}
Observe from the inclusion  $  G T \inclu   SG$ that $T \inclu G\inv SG$ and   $G\adj G T \inclu G\adj S G=: A \in~\lplus.$   Hence   $A =(G\adj G T)\adj = \tstar G\adj G$ and one has
$$
AT \inclu   (G\adj SG ) G\inv  SG  \inclu G\adj S^2 G= G\adj S (G\adj S)\adj \in \lplus.
$$
This implies that $ G\adj S^2 G= (AT)\adj = \overline{AT}  \in \lplus $ and for all $x \in \DT = \D A T = \D \tstar G\adj G T   \inclu \D \overline{AT}$ one has
\begin{align*}
  \langle {\overline{AT }}\half x , {\overline{AT }} \half x \rangle = \langle AT x ,x \rangle =  \langle \tstar G\adj G T x , x \rangle
                                     & = \langle   G\adj G T x , T x \rangle \\
                                     & \leq \n G\adj G \n \langle    T x , T x \rangle  \\
            & = \n G\adj G \n \langle   (\tstar  T)\half x , (\tstar  T)\half x \rangle.
\end{align*}
This completes the argument.
\end{proof}

Note that in the particular case where $AT =  \overline{AT} $, Theorem \ref{lasttheo} ultimately reduces to the bounded operator setting since $\DT = \D AT = \h,$ thereby justifying the inequality \eqref{inequalitychoice}. This section is ended  with a recapitalization joining the classes $\lldeux$ and $\ldeux$ together with quasi-similarity, and quasi-affinity to a bounded nonnegative operator.

 \begin{proposition}
   Let $T\in CO(\h) $ be densely defined.
 Then the following assertions  are equivalent:
   \begin{enumerate}[\rm(i)]
       \item $T$ is quasi-similar to   $S \in
           \lplus.$
        \item $T$ and $\tstar$ are quasi-affine to $S \in \lplus.$
\item $ \overline{X_1\half \tstar            X_1\mhalf} = X_1\mhalf T X_1\half \in \lplus $ and  $ \overline{X_2\half T            X_2\mhalf} = X_2\mhalf \tstar X_2\half \in \lplus $  for some quasi-affinities $X_1,X_2 \in \lplus.$
  \end{enumerate}
  In this case, there exist  $ A_i, {B_i}\inv \in \lplus$, $i=1,2$ such that  $T_1=A_1 B_1,$   $T_2= B_2 A_2 $ and
  $$
   T_1 \inclu T \inclu  T_2 .
   $$
   Moreover, $
       \overline{X_2T},
        \overline{X_1 \tstar}  \in  \lplus.$
\end{proposition}

\begin{proof}
The proof follows immediately from a combination of Lemma \ref{equivquasisimilarity} and Proposition \ref{ivaa}.
\end{proof}



%

\bibliographystyle{siam}

\end{document}